\documentclass[10pt]{article}
\usepackage[top=1in, bottom=1in, left=1in, right=1in]{geometry}
\usepackage{srcltx,graphicx}
\usepackage{amsmath, amssymb, amsthm, amsfonts}
\usepackage{color, xcolor, colortbl}
\usepackage{array, multirow}
\usepackage{subfig, overpic}
\usepackage{hyperref}
\usepackage{epstopdf}
\usepackage{algorithm}
\usepackage{algpseudocode}
\usepackage{indentfirst}
\usepackage[title]{appendix}
\usepackage[capitalize,nameinlink]{cleveref}[0.19]
\crefname{equation}{}{Equations}
\Crefname{equation}{Eq.}{Equations}
\algdef{SE}[DOWHILE]{Do}{doWhile}{\algorithmicdo}[1]{\algorithmicwhile\ #1}%

\numberwithin{equation}{section}

\newcommand{\bbR}{\mathbb{R}}

\newcommand{\cB}{{B}}
\newcommand{\cD}{\mathcal{D}}

\newcommand{\cH}{\mathcal{H}}
\newcommand{\cI}{\mathcal{I}}
\newcommand{\cL}{\mathcal{L}}
\newcommand{\cS}{\mathcal{S}}
\newcommand{\cT}{\mathcal{T}}

\newcommand\new{\mathrm{new}}
\newcommand\init{\mathrm{init}}
\newcommand\err{\mathrm{err}}
\newcommand\sps[1]{^{(#1)}}

\newcommand\innerproduct[3]{{\left\langle #1, ~#2\right\rangle_{#3}}}
\newcommand\mode[1]{{\left\| #1\right\|}}

\newcommand\pd[2]{\frac{\partial #1}{\partial #2}}
\newcommand\pdt[3]{\frac{{\partial}^{#3} #1}{\partial {#2}^{#3}}}
\newcommand\dd{{\,\mathrm{d}}}

\newtheorem{lemma}{Lemma}
\newtheorem{remark}{Remark}

\newtheorem{proposition}{Proposition}

\title{A Simple Multiscale Method for Mean Field Games}
\author{
  Haoya Li\thanks{Department of Mathematics, Stanford University,
	Stanford, CA 94305, email: {\tt lihaoya@stanford.edu}}, ~~
  Yuwei Fan\thanks{Department of Mathematics, Stanford University,
	Stanford, CA 94305, email: {\tt ywfan@pku.edu.cn}},~~
  Lexing Ying\thanks{Department of Mathematics and Institute for
  Computational and Mathematical Engineering, Stanford University,
	Stanford, CA 94305, email: {\tt lexing@stanford.edu}}
}

\begin{document}

\date{}
\maketitle 

\begin{abstract}
This paper proposes a multiscale method for solving the numerical solution of mean field
games which accelerates the convergence and addresses the problem of determining the initial guess. Starting from an approximate solution at the coarsest level, the method constructs
approximations on successively finer grids via alternating sweeping, which not only allows for the use of classical time marching numerical schemes, but also enables applications to both local and nonlocal problems. At each level, numerical relaxation is used to stabilize the iterative process. A second-order discretization
scheme is derived for higher order convergence. Numerical examples are provided to demonstrate the
efficiency of the proposed method in both local and nonlocal, 1-dimensional and 2-dimensional cases. 
\vspace*{4mm}

\end{abstract}

\section{Introduction}
Mean Field Games (MFG) theory was first introduced by Lasry and Lions in \cite{Lasry2007} and by
Huang, Caines and Malham{\'e} in \cite{huang2006} independently for studying the asymptotic behavior
of stochastic differential games as the number of agents tends to infinity.  The agents in the
system are assumed to be identical, and any individual agent has little impact on the outcome of the
game. Each individual agent aims to minimize a certain cost, and the strategy adopted is influenced
by the average of a certain function of the states of the other agents \cite{achdou2018mean}. In the
limit of an infinite number of agents, a typical MFG can be described by the following system
\begin{subequations}\label{eq:MFG}
  \begin{align}
    &\left\{ 
      \begin{aligned}\label{HJB}
        &u_t - \nu \Delta u + H(x, \nabla u) = V[m],  \\ 
        &u(t = 0, x) = u_0(x)+V_{0}[m(t=0)](x),
      \end{aligned} 
    \quad\qquad x \in \Omega\subset\bbR^d;\right.  \\[2mm]
    &\left\{
      \begin{aligned}\label{KFP}
        &-m_t - \nu \Delta m - \nabla \cdot (m\nabla_pH(x, \nabla u)) = 0, \\
        &m(t = T, x) = m_T(x), 
      \end{aligned}
    \qquad x \in \Omega \subset \bbR^d, \right.
  \end{align}
\end{subequations} 
where $m(t,x)$ stands for the density of the agents, and $u(t,x)$ is the cost function (or negative
utility). Among the two equations, the first one \eqref{HJB} is a forward Hamilton-Jacobi-Bellman
equation (HJB). The operators $V\geq 0$ and $V_0\geq 0$ encode the impact of other agents (via the
density function). On each point of its trajectory, each agent choose a velocity that locally minimizes its travel cost, and this leads to the appearance of the Hamiltonian $H(x,\nabla u)$. The second equation \eqref{KFP} is a backward Kolmogorov-Fokker-Planck equation (KFP), which results from the motion of individual agents in $\Omega$.

The existence and regularity of the solution have been studied extensively in the literature. For
the stationary case, a Bernstein integral method is used to obtain regularity bounds and existence
of smooth solutions in \cite{CIRANT20151294}. For the evolutionary case, Gagliardo-Nirenberg type
inequalities are used to develop priori bounds for the solutions when the Hamiltonian is
subquadratic (see \cite{Gomes2015}). The Hopf–Cole transformation gives an explicit way to study the
case where the Hamiltonian is quadratic (see \cite{cardaliaguet2012long}). However, this
transformation cannot be used to study superquadratic problems. In the superquadratic case, a
nonlinear adjoint method uses the adjoint equation to represent solutions by integrals with respect
to the adjoint variable and leads to priori bounds for the solutions (see \cite{Gomes2013}). For a
systematic probabilistic approach, we refer to the reference \cite{Rene}, in which the MFG system is studied in a standard stochastic control framework, and a useful analytical tool--the master equation is studied. We also refer the readers to \cite{Gueant2011} and \cite{Cardaliaguet2012Notes} for more detailed discussions of the MFGs.

In the past decade, several numerical methods for MFG have been developed. In \cite{achdou2010mean}, Achdou
and Capuzzo-Dolcetta proposed a first-order method for MFG with a first-order upwind scheme
for the spatial discretization and an implicit discretization in the time direction. This method
preserves many good properties of the MFG, for example, the convergence result proved in \cite{achdou2013mean}, but the resulting discrete system is a large-scale
nonlinear system. The Newton-Raphson method is employed for its numerical solution but the cost is
relatively high. Following the finite difference discretization in \cite{achdou2010mean}, \cite{briceno2018proximal} and \cite{briceno2019implementation} developed variational approaches for the stationary and time-dependent MFG with local couplings. The classical V-cycle and W-cycle multigrid preconditioner are used in \cite{achdou2012iterative, briceno2019implementation, andreev2017preconditioning}. In \cite{Carlini2012, Carlini2013A, Carlini2013, Carlini2015A, carlini2018fully, carlini2018discretization}, Carlini and Silva used the optimal control formula to represent the solution of the HJB equation. Using uniform
partition of the time interval, they solve the HJB equation recursively by considering a discrete
optimal control problem in every single time step. Then a mass conservative scheme is used for the
KFP equation in an alternating way to solve the whole MFG system. The convergence property is also
studied in these references. In \cite{OLIVIER2012MEAN}, Gu\'{e}ant studied a particular type of MFG system with a quadratic Hamiltonian via a change of variables given in \cite{Gueant2011}. In this way, the MFG system
becomes two coupled heat equations with similar source terms. It is then possible to construct a
scheme that yields monotonic sequences of approximate solutions. In \cite{Benamou2015,Benamou2017}, the authors studied the variational form of the MFG system with congestion penalization and investigated its duality. The variational problem is related with the solution of MFG system via the Fenchel-Rockafellar duality theorem. Since the variational problem is convex, an augmented-Lagrangian method can be applied to solve it efficiently. 

For the MFG system without the randomness of behavior of agents (i.e. without the Laplacian term),
various methods have been developed. The works
\cite{lin2018splitting,chow2018algorithm,chow2019algorithm} apply the genralized Hopf and Lax
formulas of the HJB equation and conjure the Hamiltonian in the discretization. These algorithms
have the advantage that they do not suffer from the curse of dimensionality. The work in
\cite{liu2020computational} addresses a type of MFG with nonlocal interaction (for example, when the
operator $V$ is nonlocal) by using kernel-based representations of mean-field interactions and
feature-space expansions. Most recently, neural network type methods, such as the APAC-net in
\cite{lin2020apac} and the framework based on Lagrangian method \cite{ruthotto2020machine}, have
also been applied to mean field game problems, especially in high dimensions.

\paragraph{Contributions.}
The main difficulty of solving MFG systems comes from the forward-backward nature of the
system. Numerical discretization of such a system leads to a nonlinear large system of equations. To
overcome this difficulty, we introduce a simple alternating sweeping algorithm to decouple the two equations
and to make the application of time marching schemes possible. More specifically, the method solves
the HJB equation or the KFP equation alternatingly by fixing the other quantity. This allows us to
solve the MFG system by solving two time-dependent PDEs, and the overall cost is decreased
remarkably when the convergence is guaranteed. In order to make this work, two critical issues need
to be addressed: (1) whether the algorithm converges and (2) how to choose a good initial guess?

To address the first question, we study the condition of convergence and introduce a relaxation
technique for satisfying this condition. The convergence can then be guaranteed with a proper
selection of the relaxation factor. Numerical simulations show that the relaxation technique is
highly effective.

To answer the second question, we introduce a multiscale method. More specifically, the problem is
discretized with a sequence of successively finer grids. At the coarsest level, the system is
relatively small, hence it can be solved directly by, for example, the Newton iteration in
\cite{achdou2010mean}. At each finer level, interpolating the solution from the previous level
provides a good initial guess for the system at this level, and the solution is refined using
alternating sweeping of the HJB and KFP equations. This process is repeated hierarchically until one
reaches the finest grid. This multiscale method not only provides a good initial guess for each
level, but also dramatically accelerates the convergence of the alternating sweeping.  For the
discretization of the MFG system, we also introduce a second-order scheme in order to reduce the
numerical error. More specifically, a Beam-Warming scheme is applied on the spatial discretization,
and a Crank-Nicolson scheme is introduced on the temporal discretization. The scheme is proved to
preserve the conservation of mass.

\paragraph{Contents.}
The rest of the paper is organized as follows. \cref{sec:alg} provides the main description of the
numerical scheme. The general framework of the alternating sweeping procedure and the relaxation
method is proposed in \cref{sec:as}, the multiscale algorithm is detailed in \cref{sec:ms}, and the
second-order finite difference scheme in introduced in \cref{sec:dicr}. \cref{sec:application}
studies the numerical performance of the propsoed algorithms. Finally \cref{sec:conclusion}
concludes with some discussion for future work.

\section{Numerical Algorithm}\label{sec:alg} 
\subsection{Notations}\label{notations}

To simplify the discussion, let us consider the spatial domain $\Omega=[0,1]^d$ with the periodic
boundary condition. We introduce the notations for $1$-dimensional case as it is straightforward to
extend to the $d$-dimensional case.

We partition $\Omega$ by a hierarchical uniform Cartesian mesh with step size
$h\sps{\ell}=\frac{1}{N\sps{\ell}}=2^{-\ell}$ for each level $\ell=L_0, \ldots,L$, where $L_0\leq L$
are given positive integers. For each $\ell$, the grid points are denoted as
$x_i\sps{\ell}=ih\sps{\ell}$, with $i=1,\ldots,N\sps{\ell}$. In addition, for each $\ell$, the
time interval $[0, T]$ is also uniformly partitioned with time step
$\tau\sps{\ell}=\frac{T}{N_t\sps{\ell}}=2^{-\ell}T$. The grid points in time are denoted as
$t_n\sps{\ell} = n\tau\sps{\ell}$. One can easily extend to the case where different stepsizes are
used.

For the level $\ell$ mesh, denote the approximations to $u(x_i\sps{\ell}, t_n\sps{\ell})$ and
$m(x_i\sps{\ell}, t_n\sps{\ell})$ by $u_i^{n,(\ell)}$ and $m_i^{n,(\ell)}$, respectively.  It is
often convenient to abbreviate $\left(u_i^{n,(\ell)}\right)_{i=0,\ldots,2^{\ell}-1}$ and
$\left(m_i^{n,(\ell)}\right)_{i=0,\ldots,2^{\ell}-1}$ as $u^{n,(\ell)}$ and $m^{n,(\ell)}$, and
further denote $(u^{n,(\ell)})_{n=0,\dots,N_t}$ and $(m^{n,(\ell)})_{n=0,\dots,N_t}$ by
$U\sps{\ell}$ and $M\sps{\ell}$, respectively. Notice that the lower-case notations $u$ and $m$
represent the solution on the mesh for a given time, and the upper-case notations $U$ and $M$
represent the whole solution on a given level.

As an example of the notations adopted, a backward Euler discretization of the 1-dimensional MFG
equations \cref{eq:MFG} on level $\ell$ is of the following form, for $i=1,\dots, N\sps{\ell}$,
\begin{subequations}\label{eq:nMFG}
  \begin{align}
    &\left\{
      \begin{aligned}\label{eq:nHJB}
        & \frac{u_i^{n+1, (\ell)}-u_i^{n, (\ell)}}{\tau\sps{\ell}} + (\cL u^{n+1, (\ell)})_i +
        H(x_i\sps{\ell}, {(\cD u^{n+1, (\ell)})}_i) = V[m^{n, (\ell)}]_i, 
        \quad n = 0,\dots,N_t\sps{\ell}-1, \\
        & u_i^{0, (\ell)} = u_0(x_i\sps{\ell})+V_0[m^{0, (\ell)}]_i; \\ 
     \end{aligned}\right. \\[2mm]
    &\left\{
      \begin{aligned} \label{eq:nKFP}
        &-\frac{m_i^{n, (\ell)}-m_i^{n-1, (\ell)}}{\tau\sps{\ell}} + 
        (\cL m^{n-1,(\ell)})_i - B_i(m^{n-1, (\ell)}, u^{n, (\ell)}) = 0, 
        \quad n = N_t\sps{\ell}, N_t\sps{\ell}-1,\dots,1, \\
        &m_i^{N_t\sps{\ell}, (\ell)} = m_T(x_i\sps{\ell}), 
     \end{aligned}
     \right.
  \end{align}
\end{subequations}
where $\cD$ stands for the discretization of the spatial gradient operator $\nabla$, and $\cL$
stands for the discretization of $-\nu\Delta$, a scalar multiple of the Laplace operator, and $B$ is
the discrete analog of $\nabla \cdot (m\nabla_{p}H(x,u))$. For the sake of brevity we overload the
notation $\cD$, $\cL$, $H$, $V$ and $V_0$ for the corresponding operators on grid functions on each
level. In the paper, if the level $\ell$ is clear in the context, the superscript $\ell$ will be
omitted. 

\subsection{Alternating sweeping}\label{sec:as} % big circle
Due to the forward-backward structure of system \cref{eq:nMFG}, one cannot directly apply a time
marching scheme. A simple but key observation is that if one fixes the value of $M$, the system
\cref{eq:nHJB} becomes a single forward parabolic equation with an initial condition. Similarly, if
we fix the value of $U$, the system \cref{eq:nKFP} becomes a backward parabolic equation with a
terminal condition. In this way, we arrive at a natural way to adopt time marching schemes.

More specificallly, we start from an initial guess $M_{\mathrm{init}}$, and use some time marching
scheme to solve \cref{eq:nHJB} while fixing $M$, then solve \cref{eq:nKFP} while fixing $U$, and
repeat this process until we reach a fixed point. We call this algorithm the \emph{alternating
sweeping}, which is depicted in \cref{alg:AS}.
\begin{algorithm}[ht]
  \caption{Alternating Sweeping for the MFG}
  \label{alg:AS}
  \textbf{Input:} Initial guess of the density $M_{\mathrm{init}}$, tolerance $\epsilon$\\
  \textbf{Output:} Numerical solution $M$ and $U$
  \begin{algorithmic}
    \Function{AlternatingSweeping}{$M_{\init}$}
    \State $M \leftarrow M_{\init}$, $U \leftarrow 1$;
    \Do
    \State Fixing $M$ and solving \cref{eq:nHJB} gives $U_{\new}$;
    \State Compute $\err_U = \mode{U_{\new}-U}/\mode{U}$;
    \State $U \leftarrow U_{\new}$; 
    \State Fixing $U$ and solving \cref{eq:nKFP} gives $M_{\new}$;
    \State Compute $\err_M = \mode{M_{\new}-M}/\mode{M}$
    \State $M \leftarrow M_{\new}$; %$U \leftarrow U_{\new}$;
    \doWhile{$\max\{\err_M, \err_U\} > \epsilon$}
    \State \textbf{return} $M$ and $U$;
  \EndFunction
\end{algorithmic}
\end{algorithm}

Now, we briefly analyze the alternating sweeping procedure from the viewpoint of fixed point
iteration. In the following discussion, we denote $M_k$ and $U_k$ as the approximate solution of
$M$ and $U$ after the $k$-th iteration in \cref{alg:AS}. %In \cite{achdou2010mean}, several
For the clarity of notation, we rewrite the functions in \cref{eq:nMFG} as:
\begin{equation}\label{fix}
    F(U, M) = 0, \qquad G(U, M) = 0
\end{equation}
by moving everything to the left hand side.

For $k \geq 2$, in the k-th iteration of Alternating Sweeping, we start from an approximation of
$m$ on the grid points, i.e. $M_{k-1}$, and then we obtain $U_{k}$ from the equation $F(U_{k},
M_{k-1}) = 0$, and finish this step by getting $M_{k}$ from $G(U_{k},M_{k}) = 0$. For $k=1$ we start
from $M_{\mathrm{init}}$. This is a variant of the usual fixed point iteration method (see for example \cite{gautschi1997numerical}). We give a convergence result in the following proposition. 

\begin{proposition}\label{cvg}
  Assume that $F$ and $G$ are continuously differentiable, and that $F_u(U^*,M^*)$ and $G_m(U^*,M^*)$ are invertible, where $U^*$ and $M^*$ are the solution to the equation \eqref{fix}, and 
  \begin{equation}\label{jacob}
    \begin{aligned}
      F_{u} &:= \pd{F}{U}(U^*,M^*),  &  F_{m} &:= \pd{F}{M}(U^*,M^*), \\
      G_{u} &:= \pd{G}{U}(U^*,M^*),  &  G_{m} &:= \pd{G}{M}(U^*,M^*).
    \end{aligned}
  \end{equation}
  Then $M_{k}$ converges to $M^*$ locally if 
  \begin{equation}\label{rho1}
    \rho(G_{m}^{-1}G_{u}F_{u}^{-1}F_{m}) < 1,
  \end{equation}
  where $\rho$ denotes the spectral radius, and $U^{k}$ converges to $U^*$ locally if
  \begin{equation}\label{rho2}
    \rho(F_{u}^{-1}F_{m}G_{m}^{-1}G_{u}) < 1. 
  \end{equation}
  Moreover, the conditions \eqref{rho1} and \eqref{rho2} are equivalent, and when they hold, 
  the convergence rate of \cref{alg:AS} is
  \begin{equation}
    r = \rho(G_{m}^{-1}G_{u}F_{u}^{-1}F_{m}) = \rho(F_{u}^{-1}F_{m}G_{m}^{-1}G_{u}).
  \end{equation}
\end{proposition}

In order to prove this proposition, we need the following lemma on the spectral radius of matrices. 
\begin{lemma}\label{lemma:specrad}
  \begin{equation}
    \forall A, B \in \mathbb{C}^{n\times n}, \quad \rho(AB) = \rho(BA).
  \end{equation}
\end{lemma}
The proof of this lemma is given in \cref{proof:specrad}. Now we can prove proposition~\ref{cvg}. 

\begin{proof}
Take the Taylor expansion of $F$ and $G$ at $U^*$ and $M^*$ we get 
\begin{equation}
  \begin{aligned}
    F(U_{k}, M_{k-1}) &= F(U^*, M^*) + F_u(U_{k} - U^*) + F_{m}(M_{k-1}-M^*) + o(|U_{k} - U^*|+|M_{k-1}-M^*|), \\
    G(U_{k},M_{k}) &= G(U^*,M^*) + G_{u}(U_{k} - U^*) + G_{m}(M_{k}-M^*) + o(|U_{k} - U^*|+|M_{k}-M^*|).
  \end{aligned}
\end{equation}

Thus we have
\begin{equation}\label{eq:diffmap}
  M_{k}-M^* = G_{m}^{-1}G_{u}F_{u}^{-1}F_{m}(M_{k-1}-M^*) + o(|M_{k-1}-M^*|). 
\end{equation}
Hence $M_{k}$ converges to $M^*$ locally with rate $\rho(G_{m}^{-1}G_{u}F_{u}^{-1}F_{m})$ if \eqref{rho1} holds (see for example \cite{gautschi1997numerical}). Similarly, $U^{k}$ converges to $U^*$ locally with rate $\rho(F_{u}^{-1}F_{m}G_{m}^{-1}G_{u})$ if \eqref{rho2} holds. 

By Lemma~\ref{lemma:specrad}, we have $\rho(G_{m}^{-1}G_{u}F_{u}^{-1}F_{m})=\rho(F_{u}^{-1}F_{m}G_{m}^{-1}G_{u})$, thus the conditions \eqref{rho1} and \eqref{rho2} are equivalent, which closes the proof. 
\end{proof}

\begin{remark}
  One can also first solve $M$ with fixed $U$ and then solve $U$ with fixed $M$. 
  Due to the upper analysis, the convergence rate of the corresponding scheme keeps the same.
\end{remark}

The upper discussion only guarantees the convergence of the alternating sweeping
method when $M_{\mathrm{init}}$ is sufficiently close to the real solution and the condition in
\cref{cvg} ($\rho(G_{m}^{-1}G_{u}F_{u}^{-1}F_{m})<1$) is satisfied. Thus, in order to make the
alternating sweeping work, one needs to address two critical issues: (1) ensure that the spectral
condition is satisfied (2) choose the initial guess $M_{\mathrm{init}}$ carefully. The relaxation
method in \cref{sec:relax} is concerned with (1) and and a multiscale method in \cref{sec:ms}
addresses (2).

\subsubsection{Relaxation}\label{sec:relax}
As a fixed point iteration, the alternating sweeping converges locally if the spectral radius of the
Jacobian matrix at the fixed point is smaller than $1$, as is clarified in \cref{cvg}.  However,
conditions \eqref{rho1} and \eqref{rho2} may not be satisfied. In order to address this, we propose
a relaxation technique to improve the convergence of \cref{alg:AS}. More specifically, we use
\begin{equation} \label{eq:relax}
  \begin{aligned}
    U&\leftarrow \alpha U_{\mathrm{new}} + (1-\alpha)U_{\mathrm{old}}, &
    M&\leftarrow \alpha M_{\mathrm{new}} + (1-\alpha)M_{\mathrm{old}},
  \end{aligned}
\end{equation}
when updating $U$ and $M$ in \cref{alg:AS} instead of $U\leftarrow U_{\mathrm{new}}$ and $M\leftarrow M_{\mathrm{new}}$, where $\alpha\in (0,1]$ is a relaxation factor. 

To see how the relaxation technique helps the condition in \cref{cvg} on spectral radius, we state the following proposition. 

\begin{proposition}\label{prop:relax}
  Assume that $F$ and $G$ are continuously differentiable, and that $F_u(U^*,M^*)$ and $G_m(U^*,M^*)$ are invertible, where $U^*$ and $M^*$ are the solution to the equation \eqref{fix}, and $F_u$ and $G_m$ are defined in \cref{cvg}. Then relaxation \cref{eq:relax} guarantees local convergence if $\alpha$ satisfies
  \[
  0 < \alpha< 2/(1+\rho(G_{m}^{-1}G_{u}F_{u}^{-1}F_{m})), 
  \]
  as long as 
  \begin{equation}\label{relaxcondJ}
  \left|\lambda_{max}(G_{m}^{-1}G_{u}F_{u}^{-1}F_{m})\right| < 1.
\end{equation} 
  where $\lambda_{max}$ denotes the eigenvalue with the largest mode. 
\end{proposition}
\begin{proof}
When $F$ and $G$ are continuously differentiable and $F_u(U^*,M^*)$ and $G_m(U^*,M^*)$ are invertible, \cref{alg:AS} leads to the fixed point iteration $M_{k+1} = \phi(M_{k})$ for a continuously differentiable function $\phi$ by the implicit function theorem, and $\phi(M^*) = M^*$. Moreover, the Jacobian matrix $J = \pd{\phi}{M}(M^*) = G_{m}^{-1}G_{u}F_{u}^{-1}F_{m}$, which is clear from \eqref{eq:diffmap}. Now we plug in \eqref{eq:relax}, which leads to the iteration scheme
\begin{equation}\label{relax}
  M_{k+1} = \alpha\phi(M_k)+(1-\alpha)M_k.
\end{equation}
By taking Taylor expansion of $\phi$ at $M^*$, we get
\begin{equation*}
  \phi(M_k) = \phi(M^*) + J(M_k-M^*) + o(\mode{M_k-M^*}).
\end{equation*}
Plugging this equation into \eqref{relax} leads to
\begin{equation*}
  \begin{aligned}
    M_{k+1} - M^* &= \alpha(\phi(M_k) - M^*) + (1-\alpha)(M_k - M^*)\\
    &= (\alpha J + (1-\alpha)\cI)(M_k - M^*) + o(\mode{M_k-M^*}),
  \end{aligned}
\end{equation*}
where $\cI$ is the identity matrix. Thus the iteration converges locally if (see for example \cite{gautschi1997numerical})
\begin{equation}\label{relaxcondition}
  \rho(\alpha J + (1-\alpha)\cI) < 1.
\end{equation}
Notice that
\begin{equation}
  \alpha J + (1-\alpha)\cI = \cI - \alpha(\cI-J).
\end{equation}
We see that condition \eqref{relaxcondition} is satisfied for $\alpha < 2/(1+\rho(G_{m}^{-1}G_{u}F_{u}^{-1}F_{m}))$ as long as $\left|\lambda_{max}(J)\right| < 1$. By Proposition~\ref{cvg}, the local convergence of $U_k$ holds under the same condition, which closes the proof. 
\end{proof}

From Proposition~\ref{prop:relax} we see that the relaxed scheme \eqref{eq:relax} requires much weaker conditions than those required by the original scheme which are stated in Proposition~\ref{cvg}. This is in agreement with our observations in numerical experiments, i.e., in many cases where the original scheme fails to converge, the relaxed scheme \eqref{eq:relax} still converes with a sufficiently small relaxation factor $\alpha$. 

In practice, we can use different relaxation factor $\alpha$ for different iteration steps in \cref{alg:AS}. For example, we can start with a small $\alpha$, and when $M_k$ and $U_k$ are close enough to the fixed point, say, $\max\{\err_M, \err_U\} < 5\epsilon$, we then use a larger $\alpha$, or even set $\alpha = 1$. In practice, this adaptive choice of $\alpha$ is able to further accelerate the convergence of our method.

\subsection{Multiscale algorithm}\label{sec:ms}

How to give a proper initial guess $M_{\mathrm{init}}$ is critical for the \cref{alg:AS}. 
A naive way is to set 
\begin{equation}\label{eq:naive}
  m_i^{n} = m_i^{N_t} = m_T(x_i), 
  \quad n = 0, 1, \ldots, N_t-1, \quad i = 1, 2, \ldots, N.
\end{equation}
However, this initial guess can be far away from the real solution, which can lead to more
alternating sweeping steps and longer computation time, or even failure of convergence.

We address the selection of $M_{\mathrm{init}}$ by using a multiscale method. In the first step, we
solve the equations \cref{eq:nMFG} on the coarsest grid with certain numerical methods (we will
explain this choice more specifically later). At each finer level, an initial guess of the solution
is obtained by interpolating the approximate solution from the previous level. This process is
repeated until the finest grid. In particular, the initial guess on the finest grid
$M_{\mathrm{init}}^{L}$ is given by interpolating $M^{L-1}$, which can be much better than the naive
initial guess \eqref{eq:naive}. The pseudocode of this multiscale method is summarized in
\cref{alg:MS}, where $M^{(\ell)}$, $U^{(\ell)}$ denote the approximate solution of $M$ and $U$ on
the $\ell$-th level grid, respectively.

\begin{algorithm}[h]
  \caption{\label{alg:MS} Multiscale algorithm for MFG}
  \textbf{Input:} Initial guess of the density $M^{(L_0)}_{\init}$ on coarsest grid $\ell=L_0$\\
  \textbf{Output:} Numerical solution $M^{(L)}$ and $U^{(L)}$ of MFG \cref{eq:nMFG} on finest grid level $L$
  \begin{algorithmic}
    \Function{Multiscale}{$M^{(L_0)}$}
    \State Solving \eqref{eq:nMFG} on the $L_0$ grid by a given method with initial guess
    $M_{\mathrm{init}}$ gives $M\sps{L_0}$ and $U\sps{L_0}$
    \For {$\ell$ from $L_0+1$ to $L$}
    \State Interpolating $M^{(\ell-1)}$ by an interpolation method gives $M^{(\ell)}_{\init}$;
    \State $U^{(\ell)}$, $M^{(\ell)}$ =\Call{AlternatingSweeping}{$M^{(\ell)}_{\mathrm{init}}$};
    \EndFor
    \State \textbf{return} $M\sps{L}$ and $U\sps{L}$;
    \EndFunction
  \end{algorithmic}
\end{algorithm}

\begin{remark}
  This algorithm is not a multigrid type method. In each step, one moves from a coarse grid to a
  fine grid and then apply the Alternating Sweeping method, and never goes back to the coarse grid
  after that.
\end{remark}

\paragraph{Numerical method on the coarsest grid}
Numerical results show that, when the multiscale method is used to provide the initial guess, the
alternating sweeping is significantly accelerated and the time spent on the multiscale hierarchy for
constructing the intial guess is negligible compared to the time saved. Since the number of
discretization points on the coarsest grid is relatively small, we are able to employ methods that
are more stable and possibly more expansive. The Newton iteration methods in \cite{achdou2010mean}
is a candidate. Another choice is the alternating sweeping method with a sufficiently small
relaxation factor.

\paragraph{Interpolation method}
For the interpolation method in \cref{alg:MS}, there are many candidates. A simple one is the
linear interpolation
\begin{equation}
  u^{2n+1,(\ell)}_{2i+1} = \frac{1}{4}
  \left(u^{n, (\ell-1)}_{i}+u^{n+1, (\ell-1)}_{i}+u^{n, (\ell-1)}_{i+1}+u^{n+1, (\ell-1)}_{i+1}\right).
\end{equation}
Other interpolation methods can be adopted as well, such as the cubic interpolation or spline
interpolation. Higher order interpolation methods introduce less interpolation error, but at the
cost of increasing computational time.

\subsection{Temporal and spatial discretization}\label{sec:dicr}

It is worth emphasizing that the alternating sweeping method in \cref{alg:AS} and the multiscale
algorithm in \cref{alg:MS} are independent on the specific temporal and spatial discretization used in
\cref{eq:nMFG}. Various finite difference schemes can be applied, and it is also possible to use
other types of methods to solve the HJB equation and KFP
equation: for example, optimal control type methods for solving the equation \eqref{HJB} when fixing the value of $M$. Below we introduce a new second-order finite difference scheme for discretizating the MFG systems.

\subsubsection{Properties of the MFG system}\label{sec:analysis}
In the MFG system \cref{eq:MFG}, the two equations are coupled through the Hamiltonian, the operator
$V[m]$, and the initial and terminal condition. Among them, the coupling through Hamiltonian is
critical. Precisely speaking, the Hamiltonian $H(x, \nabla u): \Omega \times {\bbR}^{d} \rightarrow
\bbR$ is the nonlinear term with respect to $u$ in \eqref{HJB}. Its gradient with respect to $p$,
i.e. $\nabla_pH$ is the nonlinear term in \eqref{KFP}. For the well-posedness of \cref{eq:MFG} in
continuous case, it is essential that
\begin{equation}\label{eq:weak}
  \innerproduct{\nabla \cdot (m\nabla_pH)}{w}{} = -\innerproduct{m\nabla_p H}{\nabla w}{},
  \quad w\in H^1(\Omega)
\end{equation}
holds, where the inner product $\innerproduct{f}{g}{} = \int_{\Omega}f\cdot g\dd x$, both for scalar
functions and vector functions. The linearized Hamiltonian with respect to $p$ is $\nabla_pH(x,
\nabla u) \cdot \nabla u$, and the operator $ u \rightarrow \nabla_pH(x, \nabla u) \cdot \nabla u $ is the adjoint of
the operator $ m \rightarrow - \nabla \cdot (m\nabla_pH(x, \nabla u)) $ because
\begin{equation}\label{eq:adjoint}
  \begin{aligned}
    \innerproduct{-\nabla \cdot (m\nabla_pH(x, \nabla u))}{u}{} &= \innerproduct{m\nabla_p H(x, \nabla u)}{\nabla u}{} = \innerproduct{m}{\nabla_pH(x, \nabla u) \cdot \nabla u }{} \\
    % &= \innerproduct{m}{\nabla_pH(x, \nabla u) \cdot p}{}.
  \end{aligned}
\end{equation}
This property is key to the proof of uniqueness in continuous case \cite{Lasry2007}. 

In the proof of uniqueness of
the solution, we consider any two solutions $\left(u_{1}, m_{1}\right)$ and $\left(u_{2},
m_{2}\right)$, subtract the HJB equations satisfied by them, multiply it by $m_1-m_2$, and then
subtract the KFP equations satisfied by them, multiply it by $u_1-u_2$, and then subtract these two
parts. The resulting expression is the sum of $4$ nonnegative terms, and is equal to $0$, and it
turns out that $\left(u_{1}, m_{1}\right)$ must be the same with $\left(u_{2}, m_{2}\right)$. In
order for this argument to be valid, we must use the convexity of $H$ with respect to $\nabla u$,
and only under the condition \eqref{eq:adjoint} can we obtain the form of the first-order condition
of convexity of $H$. More details can be found in \cite{Lasry2007}.

In \cite{achdou2010mean}, a first-order finite difference scheme is derived using the discrete
analog of \cref{eq:weak} and its well-posedness is proved. Often, first-order schemes enjoy nice properties such as non-negtivity preserving of $M$, etc., but they can be inefficient, as they require small temporal and spatial steps even for relatively low accuracy. This problem can be serious especially in 2-dimensional or 3-dimensional cases. Below, we derive a
second-order scheme, where a discrete version of \cref{eq:weak} is used to handle the coupling
caused by the Hamiltonian. For simplicity, the following derivation assumes the spatial dimension
$d=1$, but one can extend it to multi-dimensional cases without any difficulty. 

\subsubsection{Spatial discretization}

In the following derivation, let $W$ be an arbitrary grid function. The spatial discretization of
the gradient operator $\nabla$ follows the second-order upwind scheme (also called Beam-Warming
scheme, see \cite{Beam1978An}). We introduce two one-sided second-order difference operators
\begin{equation}%\label{eq:upwind}
  (\cD W)_i^- = \frac{3W_{i} - 4W_{i-1} + W_{i-2}}{2h},\quad
  (\cD W)_i^+ = \frac{- 3W_{i} + 4W_{i+1} - W_{i+2}}{2h},
\end{equation}
denote $(\cD W)_i = ( (\cD W)_i^-, (\cD W)_i^+)$, and let the discrete analog of Hamiltonian $H$
(also denoted as $H$ for notational convenience) be
\begin{equation}
  H(x, \nabla u)|_{x_i} \rightarrow   H(x_i, (\cD U)_i^-, (\cD U)_i^+) =  H(x_i, (\cD U)_i).
\end{equation}
The explicit expression of the discrete Hamiltonian $H(x_i, (\cD U)_i)$ depends on the form of the
original Hamiltonian. For instance, when $H(x, \nabla u) = \phi(x) + |\nabla u|^{\gamma},\gamma\geq
0$,
\begin{equation} \label{eq:gam}
  \cH(x_i, (\cD U)_i) = \phi(x_i) + {\left(\sqrt{\max(-(\cD U)_i^-, 0)^2 + \max( (\cD U)_i^+, 0)^2}\right)}^{\gamma}.
\end{equation}
For the Laplace operator, the three-point central difference is
\begin{equation}\label{eq:laplace}
  (-\Delta u)_i \rightarrow (\frac{1}{\nu}\cL u)_i := \frac{u_{i+1}-2u_{i}+u_{i-1}}{h^2}.
\end{equation}

Recall that in Section \ref{notations} $B_i(m, u)$ is introduced as the discretization of
$\nabla\cdot(m\nabla_p H(x,\nabla u))$. According to the analysis of the Hamiltonian in
\cref{sec:analysis}, we replace this term with the discrete analog of \cref{eq:weak},
\begin{equation}
  \sum_{i = 1}^{N} \cB_i(m,m)W_i = -\sum_{i = 1}^{N} m_i (\nabla_pH(x_i, (\cD u)_i)\cdot (\cD W)_i),
\end{equation}
where $\nabla_pH(x_i, (\cD u)_i)$ denotes the gradient of the discrete Hamiltonian $H$ with respect
to $\cD u$. Comparing the coefficient of $W_i$ on both sides of the equation leads to the expression
of $\cB_i(m,u)$ as
\begin{equation}%\label{eq:secondorderB}
  \begin{aligned}
    \cB_{i}(m,u) &= \frac{1}{2h}\left[{3m_{i}\left(\pd{H}{p_1}\right)_{i} 
      -  4m_{i-1}\left(\pd{H}{p_1}\right)_{i-1}  
      + m_{i-2}\left(\pd{H}{p_1}\right)_{i-2}} \right. \\
      &\qquad\qquad \left.{- 3m_{i}\left(\pd{H}{p_2}\right)_{i} 
      + 4m_{i+1}\left(\pd{H}{p_2}\right)_{i+1} 
    - m_{i+2}\left(\pd{H}{p_2}\right)_{i+2} }\right], \\ 
  \end{aligned}
\end{equation}
where $\frac{\partial H}{\partial p_1}$ is the partial derivative of $H$ with respect to $(\cD
u)^-$, and $\frac{\partial H}{\partial p_2}$ is the partial derivative of $H$ with respect to $(\cD
u)^+$. Finally, a semi-discrete version of MFG with respect to $x$ is given as follows:
\begin{equation}
  \begin{aligned}
    \pd{u_i}{t} &+ (\cL u)_i + H(x, (\cD u)_i) = V[M]_i, \quad i=1,\dots,N,\\
    -\pd{m_i}{t} &+ (\cL m)_i - B_i(m, u) = 0, \quad i=1,\dots,N.
  \end{aligned}
\end{equation}

\subsubsection{Temporal discretization}
The specific expression of spatial discretization is not important in the following discussion. To
simplify the discussion, let us introduce
\begin{subequations}
  \begin{align}
    \cS(u, m)_i &= -(\cL u)_i - H(x, (\cD u)_i) + V[m]_i, \quad i=1,\dots,N,\\
    \cT(u, m)_i &= (\cL m)_i - B_i(m,u), \quad i=1,\dots,N.
  \end{align}
\end{subequations}
Then the Crank-Nicolson scheme can be written as
\begin{subequations}\label{CN}
  \begin{align}
    \frac{u_i^{n+1}-u_i^n}{\tau} &= \frac{1}{2}\left( \cS(u^n,m^n)_i+\cS(u^{n+1},m^{n+1})_i
    \right), \label{HJBcn}\\
    \frac{m_i^{n+1}-m_i^n}{\tau} &= \frac{1}{2}\left( \cT(u^n,m^n)_i+\cT(u^{n+1},m^{n+1})_i
    \right). \label{KFPcn}
  \end{align}
\end{subequations}
Together with the initial condition and terminal condition
\begin{subequations}
  \begin{align}
    u_i^0 &= u_0(x_i)+V_0[m^0]_i, \quad i = 1, 2, \ldots, N,  \\ 
    m_i^{N_t} &= m_T(x_i), \quad i = 1, 2, \ldots, N,
  \end{align}
\end{subequations}
the equation \cref{eq:MFG} is fully discretized as follows 
\begin{subequations}\label{MFGfd}
  \begin{align}
    &\begin{aligned}\label{HJBfd}
       \frac{u_i^{n+1}-u_i^n}{\tau} &=  \frac{1}{2}(- {(\cL (u^n + u^{n+1}))}_i - H(x_i, {(\cD u^n)}_i) - H(x_i, {(\cD u^{n+1})}_i) + V[m^n]_i+V[m^{n+1}]_i),\\
       u_i^0 &= u_0(x_i)+V_0[m^0]_i, \quad 
       i = 1, \ldots, N, n = 0, 1, 2, \ldots, N_t - 1, \\ 
     \end{aligned}\\
    &\begin{aligned} \label{KFPfd}
       \frac{m_i^{n+1}-m_i^{n}}{\tau} &= \frac{1}{2}({(\cL (m^n + m^{n+1}))}_i - B_i(m^n, u^n) - B(m^{n+1}, u^{n+1})), \\
       m_i^{N_t} &= m_T(x_i), \quad 
       i = 1, \ldots, N, n = 0, 1, 2, \ldots, N_t - 1. \\
     \end{aligned}
  \end{align}
\end{subequations}

The following propositions summarize the properties satisfied by this dicretization scheme.
\begin{proposition}\label{prop:order}
  Assume that the real solution $u$ and $m$ are at least fourth-order differentiable and the
  operator V is defined pointwise, then the local truncation error of scheme \eqref{MFGfd} is
  $O(\tau^2 + h^2)$.
\end{proposition}
\begin{proposition}\label{prop:mass}
  The total mass $ h\sum_{i}{m_i} $ is conserved in scheme \eqref{MFGfd}.
\end{proposition}
The proof of these two propositions are given in \cref{proof:order,proof:mass}, respectively.

\subsubsection{Solvers}
The discretization \eqref{MFGfd} derived above can be used to derive the solvers of the HJB equation
and KFP equation for the MFG. The discrete KFP equation \eqref{KFPfd} is linear with respect to $M$,
if $U$ is fixed, but the HJB equation is nonlinear with respect to $U$ due to the 
Hamiltonian $H$. However, by means of an inner iteration, the HJB equation can be reduced to a linear problem as well. In the inner iteration that solves for $u^{n+1}$, we use $u^n$ as the initial guess and solve the linearized (with respect to $U$) version of \eqref{HJBfd}. Notice that the solution is usually smooth with respect to time $t$, so the inner iterations can be effectively carried out by a Newton type inner iteration. Hence, we have the solvers for HJB and KFP in \cref{alg:HJB,alg:KFP}, respectively.

\begin{algorithm}[ht]
  \caption{Solving discrete HJB \eqref{HJBfd}}
  \label{alg:HJB}
  \textbf{Input:} A guess of the mass density $M=(m_i^n)$\\
  \textbf{Output:} Solution $U$ of the HJB equation \eqref{HJBfd}
  \begin{algorithmic}
    \Function{SolveHJB}{$M$}{}
    \State $u_i^0 \leftarrow u_0(x_i), \quad i = 1, 2, \ldots, N$
    \For {$n$ from $0$ to $N_t-1$}
    \State $Z^0 \leftarrow u^n$,  $k \leftarrow 0$
    \Do
    \State Evaluate $Z^{k+1}$ by solving the linear system: 
    \begin{align*}
    \qquad& \left(\cI + \frac{\tau \cL}{2} + \frac{\tau}{2}\pd{H}{U}(Z^k)\right)  Z^{k+1} =
    \left(\cI - \frac{\tau \cL}{2}\right)u^n \\
    &\qquad\qquad\qquad + \frac{\tau}{2}
    \left(V[m^n] + V[m^{n+1}]-H(x, \cD u^n)-H(x, \cD Z^k)+\pd{H}{U}(Z^k)Z^k\right)
    \end{align*}
    \State $k \leftarrow k+1$
    \doWhile{$\|Z^{k+1} - Z^{k} \| / \|Z^{k}\| < \epsilon $}
    \State $u^{n+1} \leftarrow Z^{k+1}$
    \EndFor
    \State \textbf{return} $U=(u_i^{n})$
    \EndFunction
  \end{algorithmic}
\end{algorithm}

\begin{algorithm}[ht]
  \caption{Solving discrete KFP \eqref{KFPfd}}
  \label{alg:KFP}
  \textbf{INPUT:} A guess of the value function $U$\\
  \textbf{OUTPUT:} Solution $M$ of the KFP equation \eqref{KFPfd} with given $U$
  \begin{algorithmic}
    \Function{SolveKFP}{$U$}
    \State $m^{N_t} \leftarrow m_T(x_i), \quad i = 1, 2, \ldots, N$,
    \For {$n$ from $N_t$ to $1$ by $-1$}
    \State Evaluate $m^{n-1}$ by solving the linear system: 
    \[
      \left(\cI + \frac{\tau \cL}{2} + \frac{\tau}{2}{\left(\pd{H}{U}(u^n)\right)}^T\right)m^{n-1} = 
      \left(\cI - \frac{\tau \cL}{2} - \frac{\tau}{2}{\left(\pd{H}{U}(u^{n+1})\right)}^T\right)m^{n}
    \]
    \EndFor
    \State \textbf{return} $M$
    \EndFunction
  \end{algorithmic}
\end{algorithm}

\begin{remark}\label{re:inner}
  In the inner iteration of \cref{alg:HJB}, $u^n$ is chosen as the initial value for solving
  $u^{n+1}$. 
  If the solution $u$ is smooth with respect to $t$, the difference $u^{n+1}-u^n$ is small, thus the
  number of inner iteration in \cref{alg:HJB} is small. Numerical tests show that the average number
  of inner iterations in \cref{alg:HJB} is approximately $2$.
\end{remark}

\section{Numerical results}\label{sec:application}

This section presents several numerical examples for one-dimensional and two-dimensional cases to
illustrate the efficiency of the proposed algorithms.

\subsection{One-dimensional case}
Following the numerical setup in \cite{achdou2010mean}, the following Hamiltonian is used for
one-dimensional numerical tests:
\begin{equation}\label{eq:1dsetting}
  H(x, \nabla u) = -200\cos(2\pi x) + 10\cos(4\pi x) + |\nabla u|^{\gamma},
\end{equation}
where $\gamma \geq 0$. A greater $\gamma$ indicates stronger nonlinearity, and $\gamma=2$ corresponds
to the usual form of energy.
In all numerical tests, we set $V_0[m(t=0)](x)=0$ unless specifically stated.

\subsubsection{First-order scheme vs second-order scheme}
Let us first compare the performance of the first-order scheme and the second-order scheme.
The initial-terminal conditions are chosen as:
\begin{subequations}\label{eq:1dbd}
  \begin{align}
    u_0(x) &= \sin(4\pi x)+0.1\cos(10\pi x), \quad x \in [0, 1], \\
    m_T(x) &= 1 + \frac{1}{2}\cos(2\pi x), \quad x \in [0, 1].
  \end{align}
\end{subequations}
The potential is set as $V[m](x) = m(x)^2$, and the end time is $T = 0.01$. We set the coefficient
$\nu=0.4, \ \gamma=2$, and the relaxation factor $\alpha=1$. 

The multiscale method starts from the level $L_0 = 4$ and stops at the level $L = 13$, i.e. the size
of the finest grid is $N\sps{L} = N_t^{(L)} = 8192$. We comment here that $N$ and $N_t$ are not required to be equal. In practice, we can use different $N$ and $N_t$, and the results are similar to what we represent here. The tolerance $\epsilon$ in \cref{alg:AS} is
set as $\epsilon=1\times{10}^{-6}$ while in \cref{alg:HJB}, we set $\epsilon=1\times{10}^{-7}$.  In
all the algorithms, the norm is chosen as the $L^2$ norm.
The solution of $U$ and $M$ on the grid $N^{(10)}=N_t^{(10)}=1024$ is shown in \cref{fig:Sol1}.
One can see that in this case, the value of $U$ (indicating the negative utility) near $x=0.88$ and
$x=0.33$ is smaller at the end of the evolution, and the density of agents gathers towards the
places relatively more desirable.

\begin{figure}[ht]
  \centering
  \subfloat[Profile of $U$]{
    \includegraphics[width=0.45\textwidth]{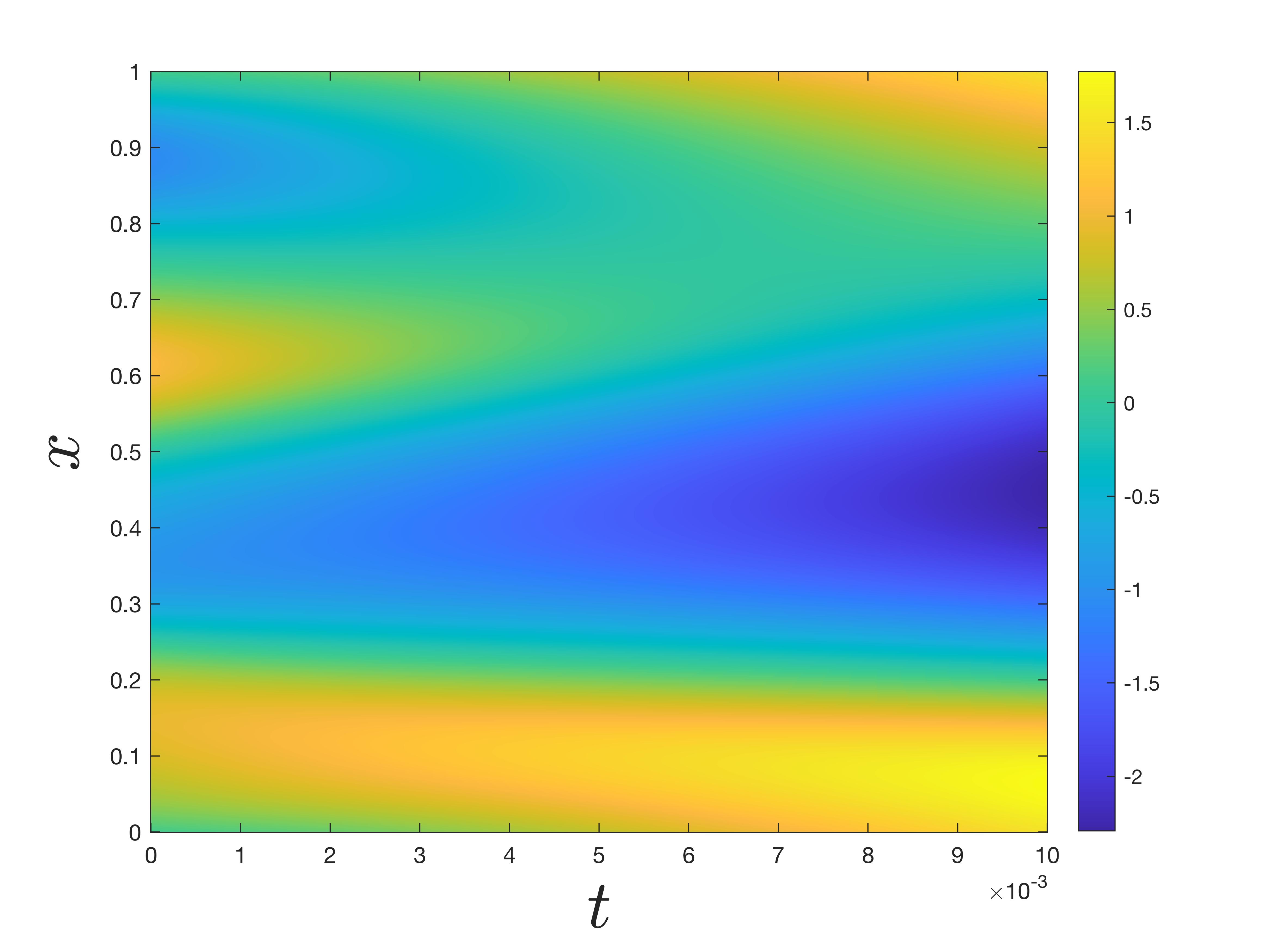}
  }
  \subfloat[Profile of $M$]{
    \includegraphics[width=0.45\textwidth]{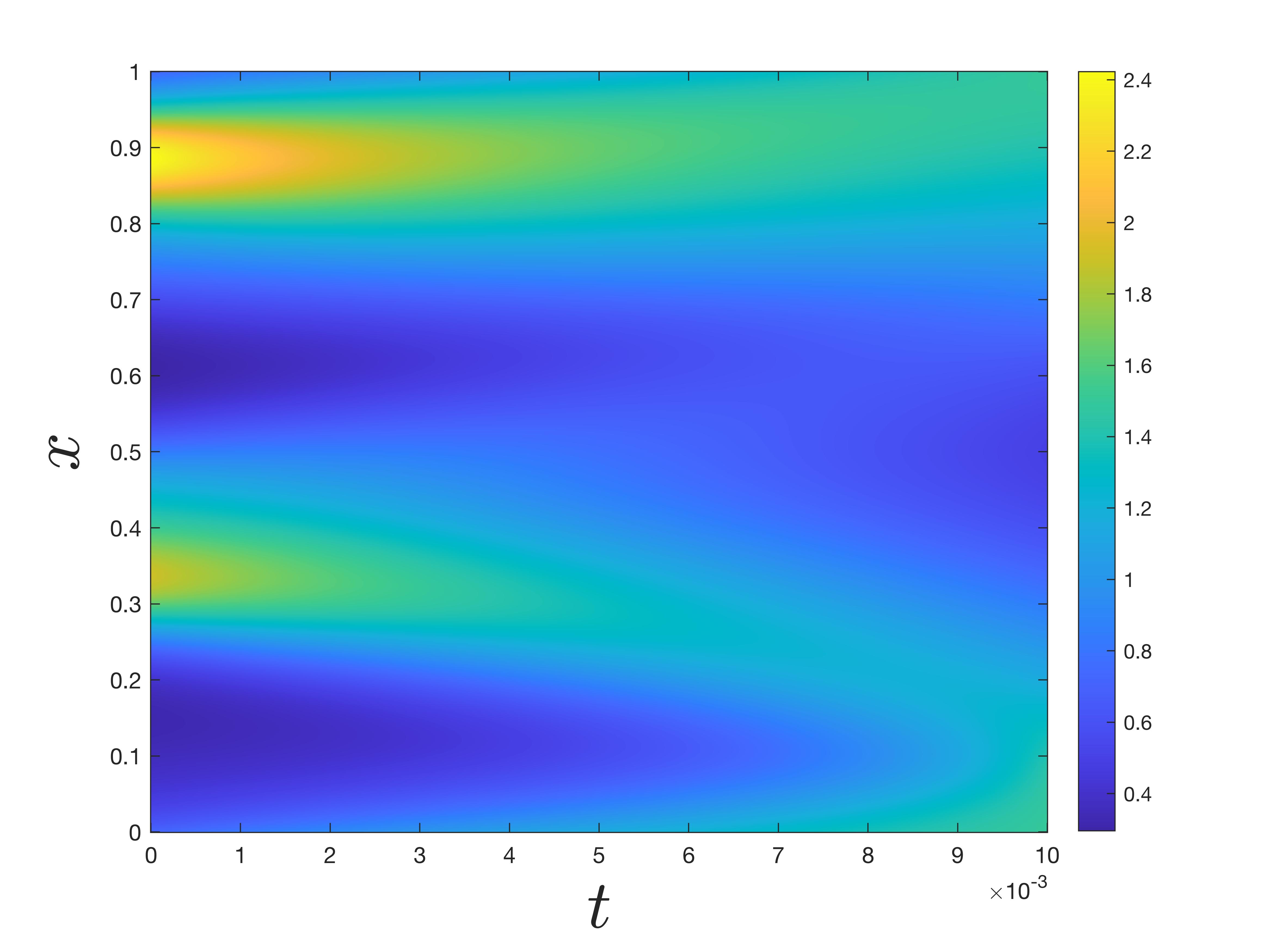}
  }
  \caption{\label{fig:Sol1}
  Profile of the solutions on the grid $N^{(10)}=N_t^{(10)}=1024$. }
\end{figure}

The convergence order can be checked by calculating the relative error of numerical solutions on
each grid in comparison to the solution on the finest grid. The numerical results are summarized in
\cref{tab:mfg1d_err}. The results show that the second-order scheme derived in this paper truly
gives rise to second-order convergence: it approaches the true solution much faster than the
first-order scheme, as expected.

\begin{table}[ht]
    \centering
    \begin{tabular}{ccccc}
        \hline
        $\ell$ & $\err_U, \err_M$(1st) & Order      & $\err_U, \err_M$(2nd) & Order      \\ \hline\hline
        4      & 5.2E-2, 9.1E-2        & ---        & 5.5E-2, 1.1E-1        & ---        \\ \hline
        5      & 3.0E-2, 5.4E-2        & 0.80, 0.77 & 1.9E-2, 3.9E-2        & 1.57, 1.54 \\ \hline
        6      & 1.6E-2, 3.0E-2        & 0.88, 0.85 & 5.3E-3, 1.1E-2        & 1.80, 1.81 \\ \hline
        7      & 8.5E-3, 1.5E-2        & 0.94, 0.92 & 1.4E-3, 2.9E-3        & 1.92, 1.92 \\ \hline
        8      & 4.3E-3, 8.0E-3        & 0.99, 0.97 & 3.6E-4, 7.3E-4        & 1.97, 2.00 \\ \hline
        9      & 2.1E-3, 3.9E-3        & 1.03, 1.02 & 9.0E-5, 1.8E-4        & 1.99, 2.00 \\ \hline
        10     & 9.9E-4, 1.9E-3        & 1.09, 1.09 & 2.2E-5, 4.5E-5        & 2.01, 2.01 \\ \hline
        \hline
    \end{tabular}
    \caption{\label{tab:mfg1d_err}The convergence order of $U$ and $M$.}
\end{table}

We check next the efficiency of inner iteration applied in \cref{alg:HJB} by counting the average
number of iterations per time step. The average numbers of iterations per time step on each grid are
shown in \cref{table:results_1D_iteration} and the result justifies \cref{re:inner}. In addition,
among all time steps, the maximum of the difference between the total mass and $1$ is
$3.00\times10^{-15}$ on the grid of size $N^{(10)}=N_t^{(10)}=1024$, which verifies the conservation
of total mass proved in \cref{prop:mass}.

\begin{table}[h]
  \begin{center}
    \begin{tabular}{ccccccccccc}
      \hline 
      $\ell$                            & 4    & 5    & 6    & 7    & 8    & 9    & 10   & 11   & 12   & 13   \\ \hline
      inner iterations / time step(1st) & 3.00 & 3.00 & 2.17 & 2.00 & 2.00 & 2.00 & 2.00 & 2.00 & 2.00 & 2.00 \\ \hline
      inner iterations / time step(2nd) & 3.00 & 3.00 & 2.20 & 2.00 & 2.00 & 2.00 & 2.00 & 2.00 & 2.00 & 2.00 \\ \hline
    \end{tabular}
  \end{center}
  \vspace{-.5cm}
  \caption{ Number of inner iterations per time step in \cref{alg:HJB}.} 
  \label{table:results_1D_iteration}
  \vspace{-.3cm}
\end{table}

\subsubsection{Alternating sweeping vs. Newton}\label{sec:ASvsNewton}
Let us now compare the computation time of the alternating sweeping without multiscale and the
Newton method presented in \cite{achdou2010mean} to demonstrate the acceleration of
alternating sweeping.

The same Hamiltonian in \eqref{eq:1dsetting} is used with the initial-terminal conditions 
\begin{subequations}\label{eq:1dbd2}
  \begin{align}
    u_0(x) &= \cos(2\pi x), \quad x \in [0, 1], \\
    m_T(x) &= 1 + \frac{1}{2}\cos(2\pi x), \quad x \in [0, 1].
  \end{align}
\end{subequations}
We also set $V[m](x) = m(x)^2$, $T=1$, $\nu=1$, $\gamma=2$, and the relaxation factor $\alpha=1$.

By taking appropriate stop criteria to get solutions of similar accuracy by the two methods, one can
compare the running times used in the computation on the same machine, listed in \cref{tab:ASvsNewton}. To check the
correctness of our numerical solutions, the residual of \eqref{eq:nMFG} are also listed.
\begin{table}[h]
    \centering
    \begin{tabular}{cccccc}
        \hline
        $N$  & $t_{Newton}$(sec) & $t_{AS}$(sec) & $t_{Newton}/t_{AS}$ & $err_{Newton}$ & $err_{AS}$ \\ \hline\hline
        32&4.4E-1&5.5E-2&8&4.3E-4&2.5E-5\\ \hline
        64 &2.0E0&1.1E-1&18&5.3E-4&1.8E-5\\ \hline
        128 &1.0E1&3.1E-1&32&6.7E-4&1.1E-5\\ \hline
        256	&7.2E1&1.0E0&72&2.2E-4&6.1E-6\\ \hline
        512	&7.6E2&3.4E0&224&5.3E-2&3.2E-6\\ \hline
        1024 &---&1.3E1&---&---&1.6E-6\\ \hline
        2048 &---&5.6E1&---&---&8.4E-7\\ \hline
    \end{tabular}
    \caption{\label{tab:ASvsNewton}The acceleration effect of the alternating sweeping and the multiscale method.}
\end{table}

Clearly, the runtime of the Newton method is significantly longer than that for the alternating
sweeping method. The ratio of time used increases rapidly as the size of grid becomes larger, and
alternating sweeping makes computation on large grids possible. When $N=512$, the alternating
sweeping method is $224\times$ faster than the Newton method.
Moreover, the residual of the alternating sweeping method is also relatively smaller than that of
the Newton method.

\begin{figure}[h]
  \centering
  \includegraphics[width=0.4\linewidth]{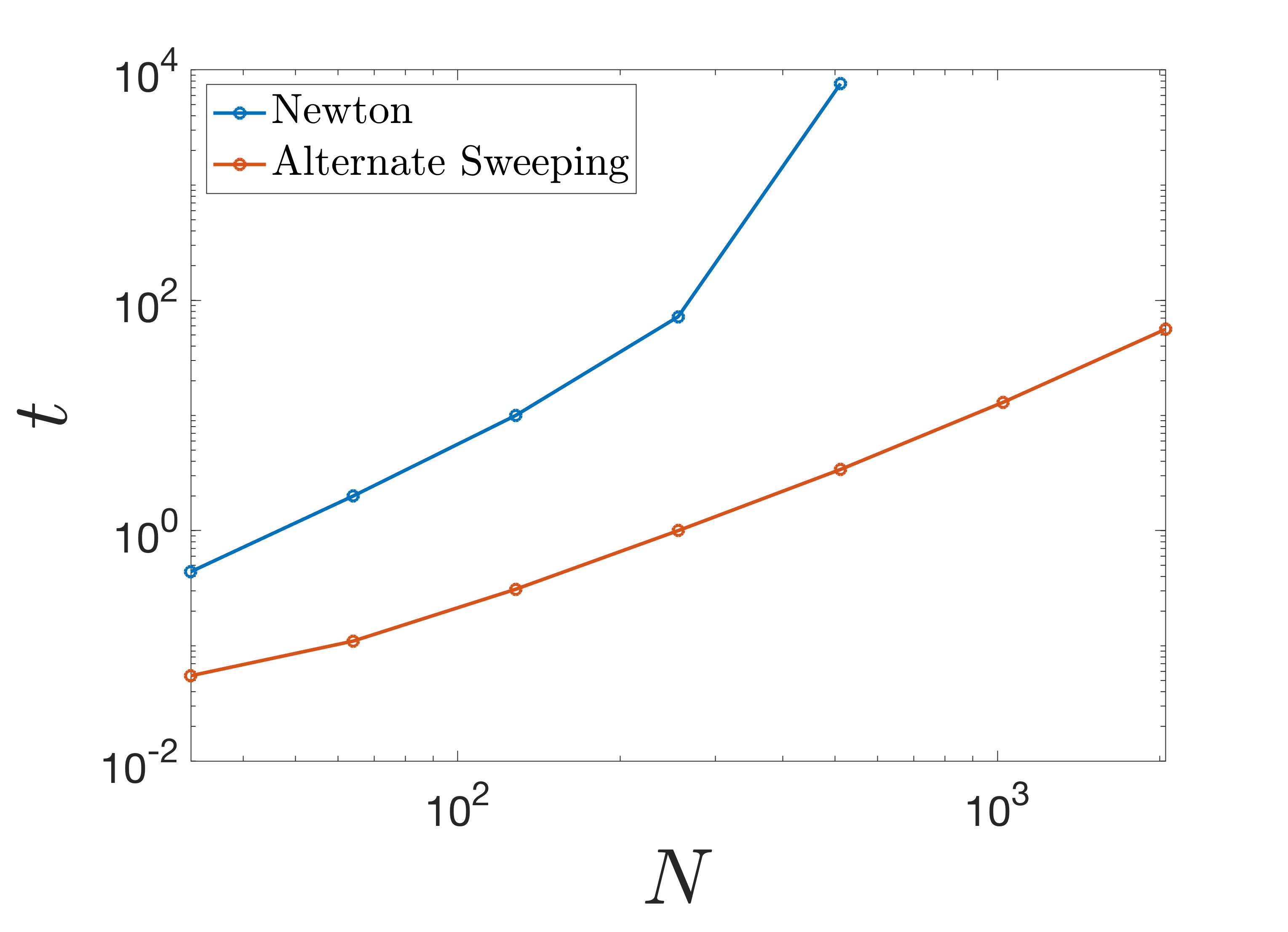}
  \caption{ \label{fig:newtonvsasvsms} Runtime of the  Newton method and alternating sweeping
  algorithm. }
\end{figure}

\subsubsection{Multiscale vs alternating sweeping}

As explained earlier, alternating sweeping only ensures local convergence, and a good initial guess
$M_{\init}$ in \cref{alg:AS} is essential. The following numerical test shows that the initial guess
given by the multiscale algorithm is significantly better than the naive initial guess
\cref{eq:naive}. Often, the algorithm fails to converge with the naive initial guess, while it
converges with the initial guess given by the multiscale method. This justifies using multiscale
method to improve the quality of initial guesses. The same Hamiltonian in \eqref{eq:1dsetting} is
used with the initial-terminal conditions
\begin{subequations}\label{eq:1dbdms}
  \begin{align}
    u_0(x) &= m_0(x)^2 + \sin(2\pi x)+0.1\cos(6\pi x), \quad x \in [0, 1], \\
    m_T(x) &= 1 + \frac{1}{2}\cos(2\pi x), \quad x \in [0, 1],
  \end{align}
\end{subequations}
and with $V[m](x) = m(x)^2$, $T = 0.01$, $\nu = 2$, and $\gamma = 2$. Notice that here the initial
condition of $u$ is not directly given as in other examples, instead it is coupled with $m_0$ via
$V_0[m_0](x) = m_0(x)^2$. This makes the naive initial guess particuarly undesirable.

Our goal is to find the solution on the grid $N^{(10)} = N_t^{(10)} = 1024$. We compare the performance of three different ways to do this.
\begin{enumerate}
\item The multiscale method that starts from a coarse grid (in this example $N^{(5)}=N_t^{(5)}=32$), and use some small $\alpha$ (in this example $\alpha = 0.2*1.1^{\ell-5}$) on the coarse grids ($5\leq l\leq9$), and with $\alpha=1$ on the grid $N^{(10)} = N_t^{(10)} = 1024$. 
\item The alternating sweeping that directly starts from the grid $N^{(10)} = N_t^{(10)} = 1024$ with the naive initial guess \cref{eq:naive} and $\alpha=1$.
\item The alternating sweeping that directly starts from the grid $N^{(10)} = N_t^{(10)} = 1024$ with the naive initial guess \cref{eq:naive} and $\alpha=0.2$.
\end{enumerate}
These three methods correspond to the red curve, the yellow curve, and the purple curve in \cref{fig:msncsry}. 

\begin{figure}[h]
  \centering
  \includegraphics[width=0.7\linewidth]{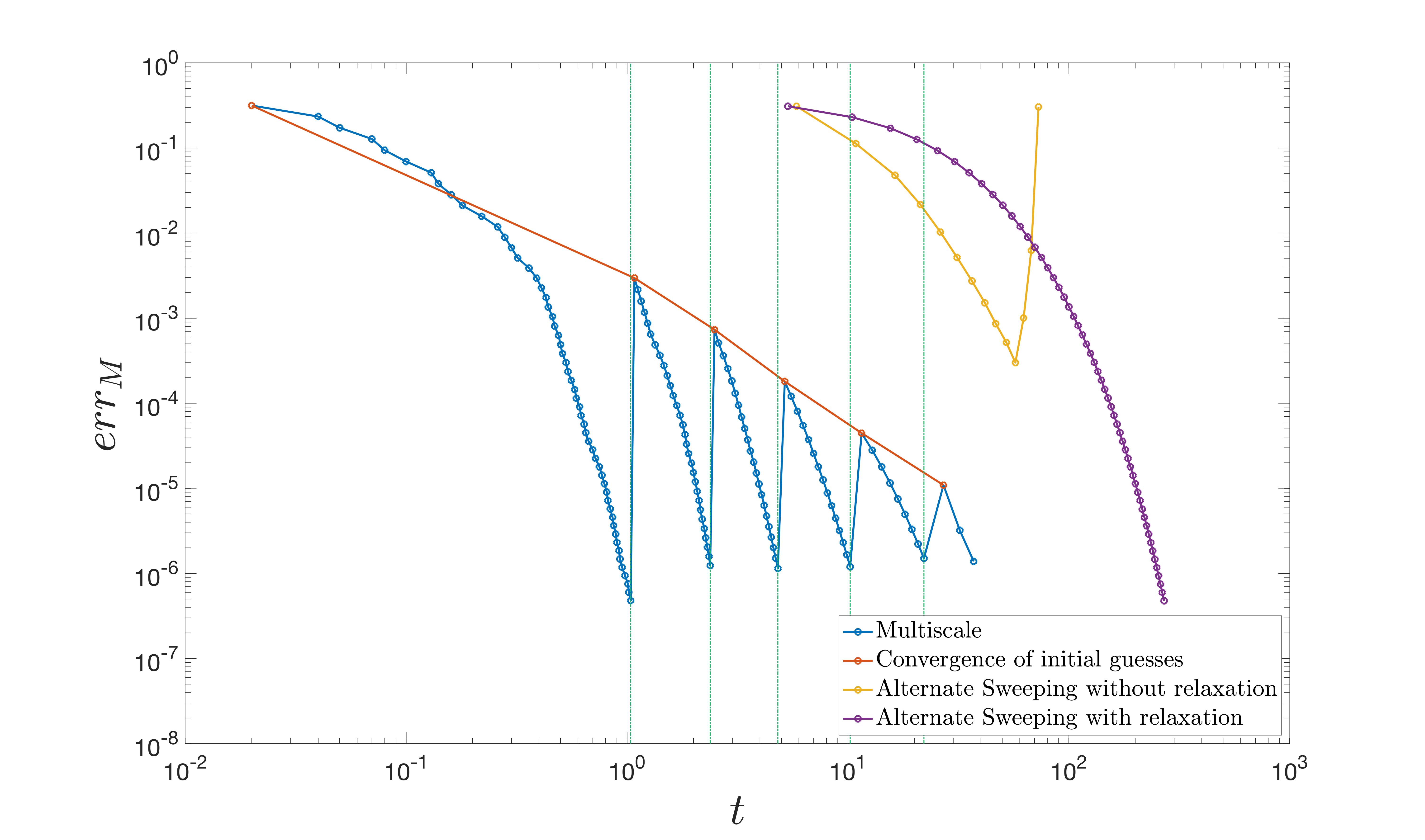}
  \caption{ \label{fig:msncsry} Runtime of the multiscale method and the alternating sweeping
  algorithm.  Each point in the plot denotes the end time of an alternating sweeping step, thus the
plots do not start from $t=0$.}
\end{figure}

In \cref{fig:msncsry}, the red curve shows that the initial guess given by the multiscale method on
each level becomes better and better and finally results in a much better initial guess than
\cref{eq:naive} on the finest mesh. As a result, the computation on the finest mesh $N^{(10)} =
N_t^{(10)} = 1024$ converges in a few iterations. However, the yellow curve fails to
converge. Comparing it with the blue curve, the reason for this failure is that the naive initial
guess is much worse than the initial guess given by the multiscale method used in the blue
curve. The purple curve converges but the time used is $270$ seconds, which is $7.3\times$ the time
used by the multiscale method on the same machine. In conclusion, compared with the alternating
sweeping without relaxation, the multiscale method shows much better convergence; compared with the
alternating sweeping algorithm with relaxation, the multiscale method is significantly faster.

\subsubsection{Performance of relaxation for stronger nonlinearity case}
As discussed in \cref{sec:relax}, one can choose sufficiently small relaxation factor $\alpha$ in
order to improve convergence when faced with significant nonlinearity (for example, when the
$\gamma$ is large in the Hamiltonian in \eqref{eq:gam}). For strong nonlinearity cases, we carry out
numerical tests with different values of $\gamma$ ranges from $3$ to $10$.  The initial-terminal conditions
used are
\begin{subequations}\label{eq:bdgam}
  \begin{align}
    u_0(x) &= \sin(4\pi x)+\frac{1}{2}\cos(10\pi x), \quad x \in [0, 1], \\
    m_T(x) &= 1 + \frac{1}{2}\cos(2\pi x), \quad x \in [0, 1],
  \end{align}
\end{subequations}
and $V[m](x) = m(x)^2$, $T = 0.01$, $\nu = 2$, $\alpha = 0.1$. The results are presented in Figure
\ref{fig:Solgam}. As the evolution of the density of agents is more visually intuitive, we only
present the solutions of $M$ here. By checking the local truncation error, we verified that all
computations converge.

\begin{figure}[h]
  \centering
  \subfloat[$\gamma=3$]{ \includegraphics[width=0.24\textwidth]{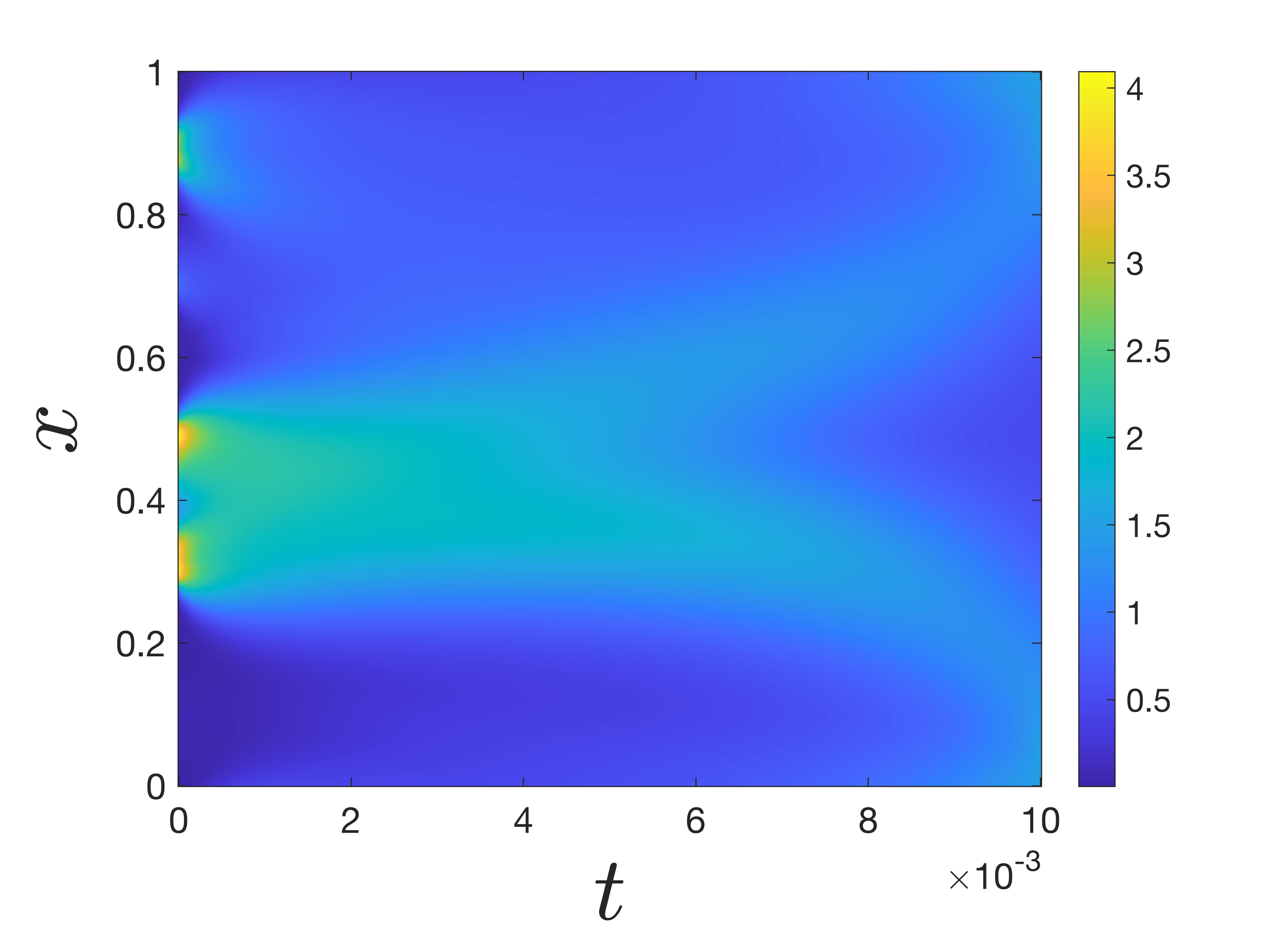} }
  \subfloat[$\gamma=4$]{ \includegraphics[width=0.24\textwidth]{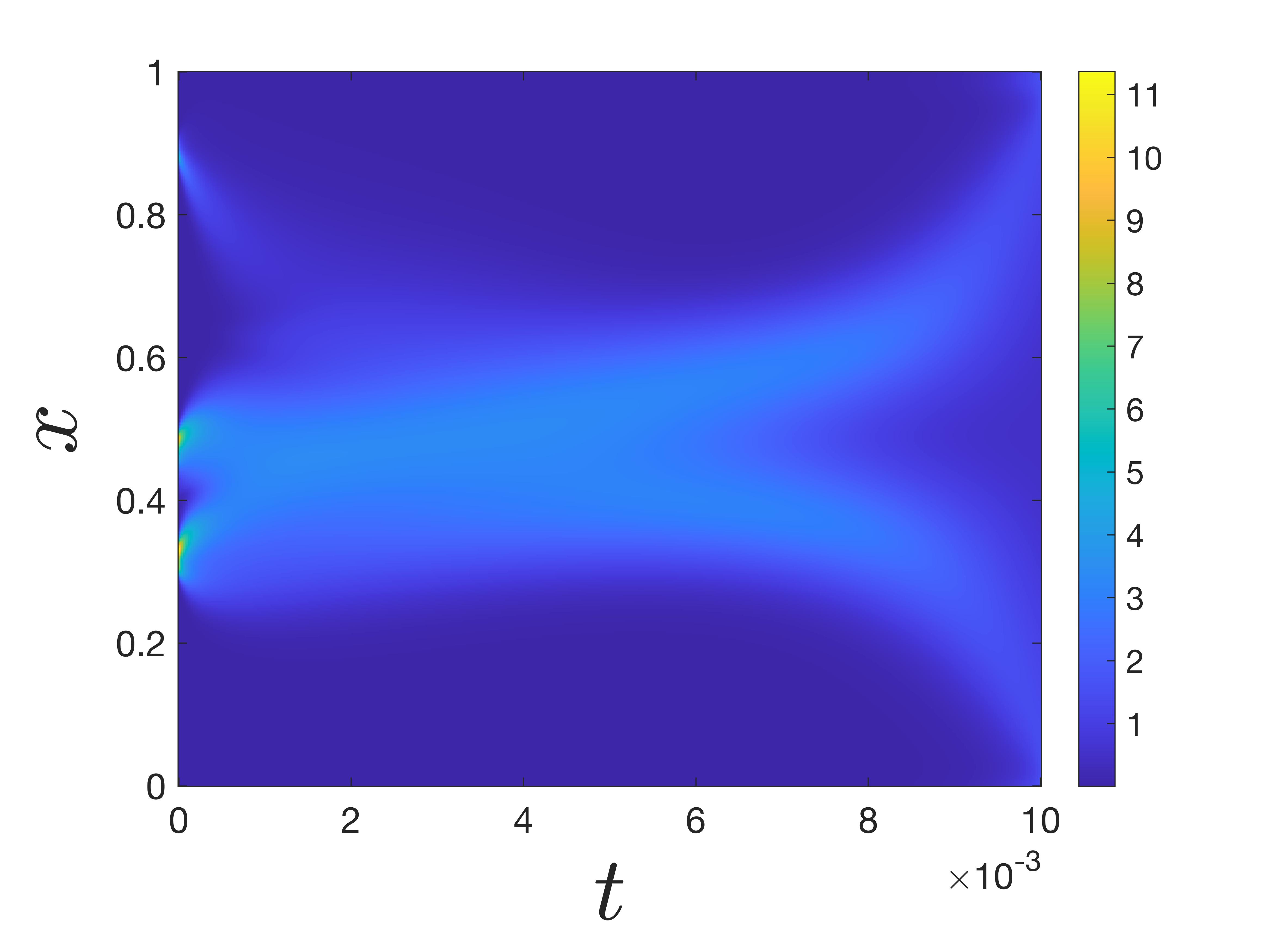} }
  \subfloat[$\gamma=5$]{ \includegraphics[width=0.24\textwidth]{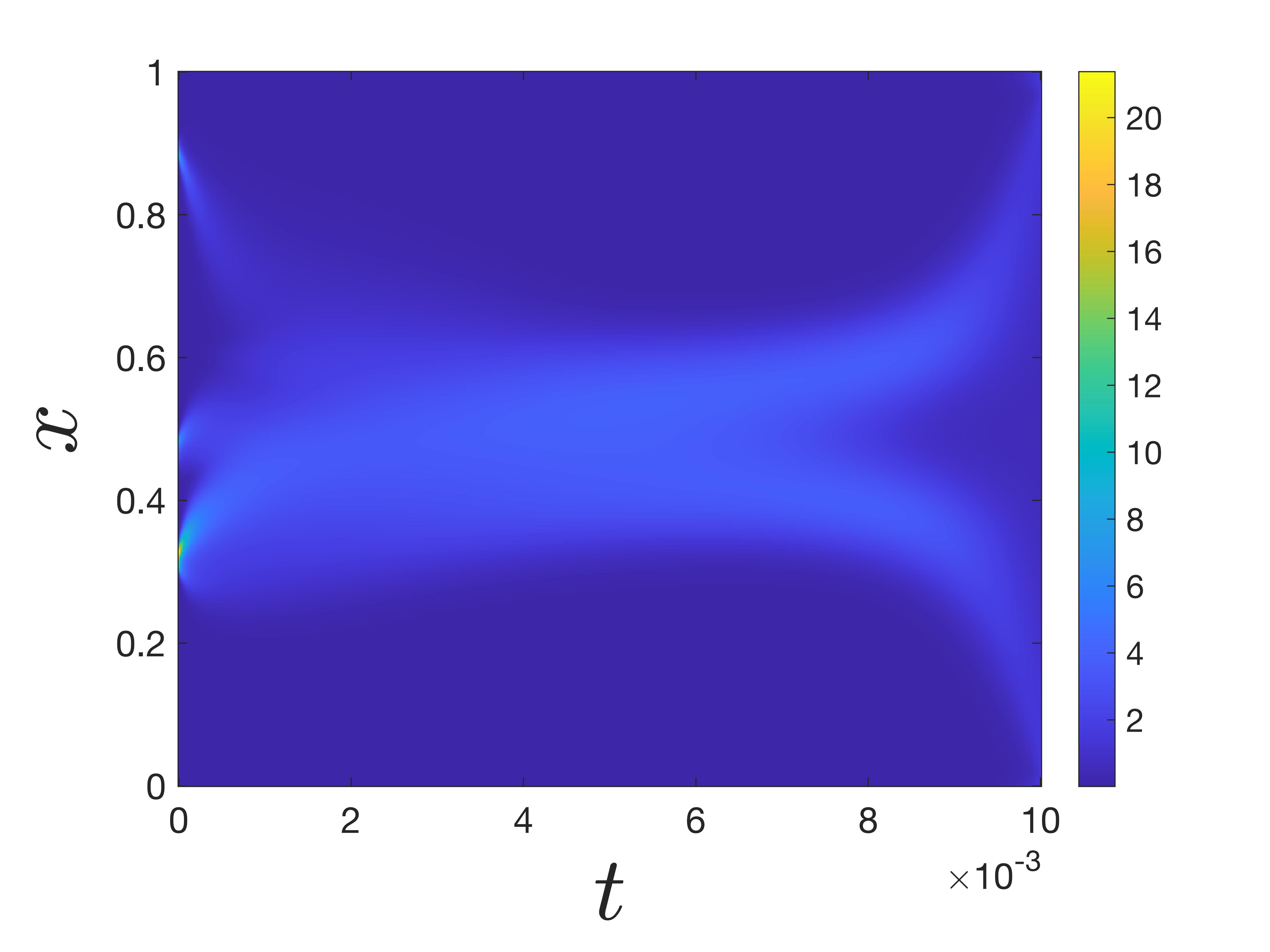} }
  \subfloat[$\gamma=10$]{ \includegraphics[width=0.24\textwidth]{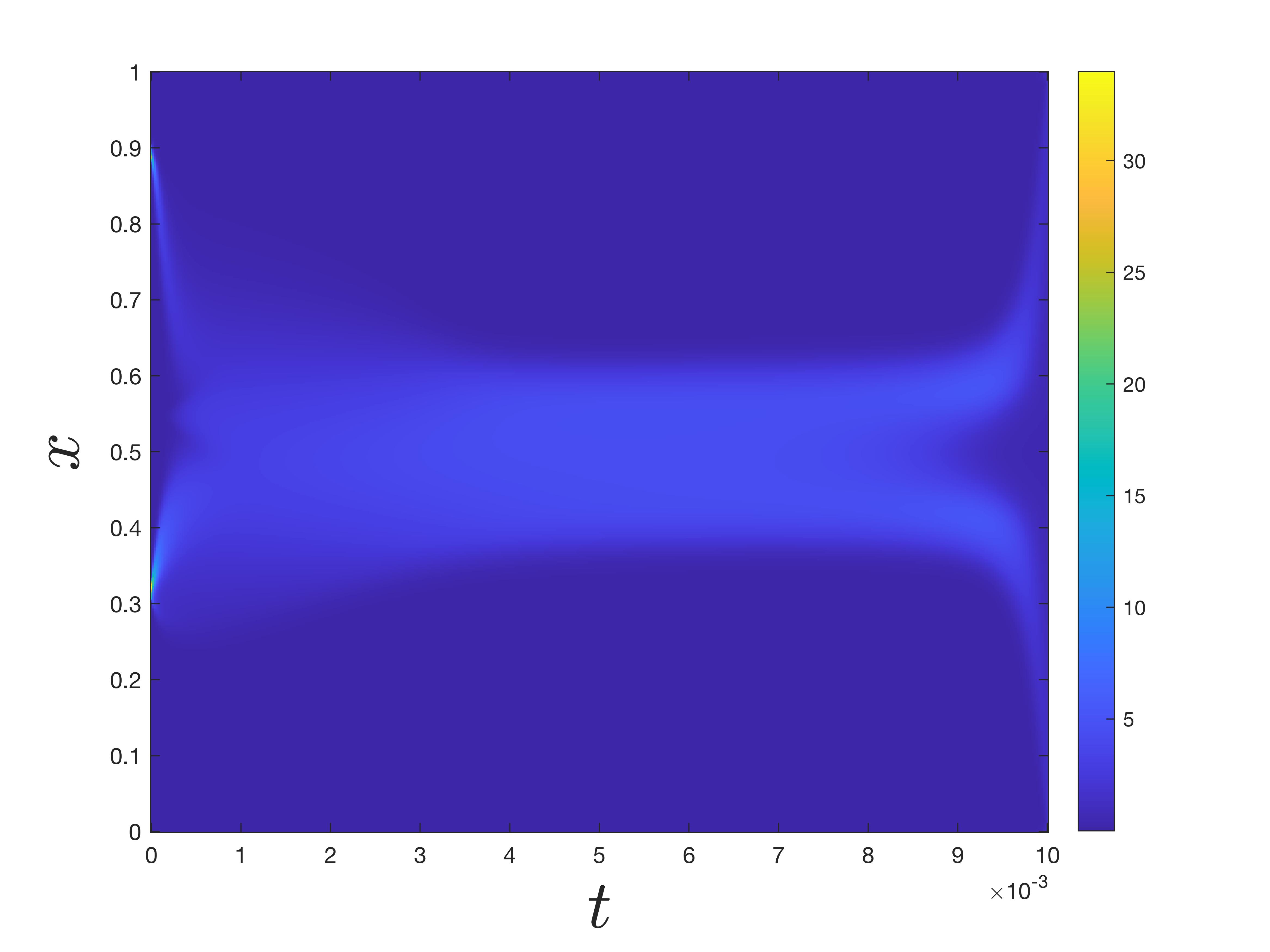} }
  \caption{\label{fig:Solgam} Profile of the solution $M$ for different $\gamma$.}
\end{figure}

\subsubsection{Performance of relaxation for the weak randomness case}

The Laplacian term in \eqref{eq:MFG} is related to the stochastic behavior of agents in, for
example, pedestrian dynamics. The deterministic limit is obtained by letting the coefficient $\nu$
go to $0$ and thus the case with small $\nu$ is quite important. When the coefficient of the Laplace
operator is small, the nonlinear term is relatively large, thus the computation usually becomes more
difficult. Similar to the cases with large $\gamma$, the computation in this case typically requires
using small relaxation factor $\alpha$. In the following tests we use the same initial-terminal conditions
as in \eqref{eq:bdgam} with $V[m](x) = m(x)^2$, $T = 0.01$, and $\gamma = 2$. Figure \ref{fig:Solnu}
summarizes the results of $\nu = 0.2, 0.1, 0.05, 0.02$, with the corresponding $\alpha$ equal to
$0.5, 0.2, 0.1, 0.05$. The result shows the process of approaching deterministic limit.

\begin{figure}[h]
  \centering
  \subfloat[$\nu = 0.2 $]{ \includegraphics[width=0.24\textwidth]{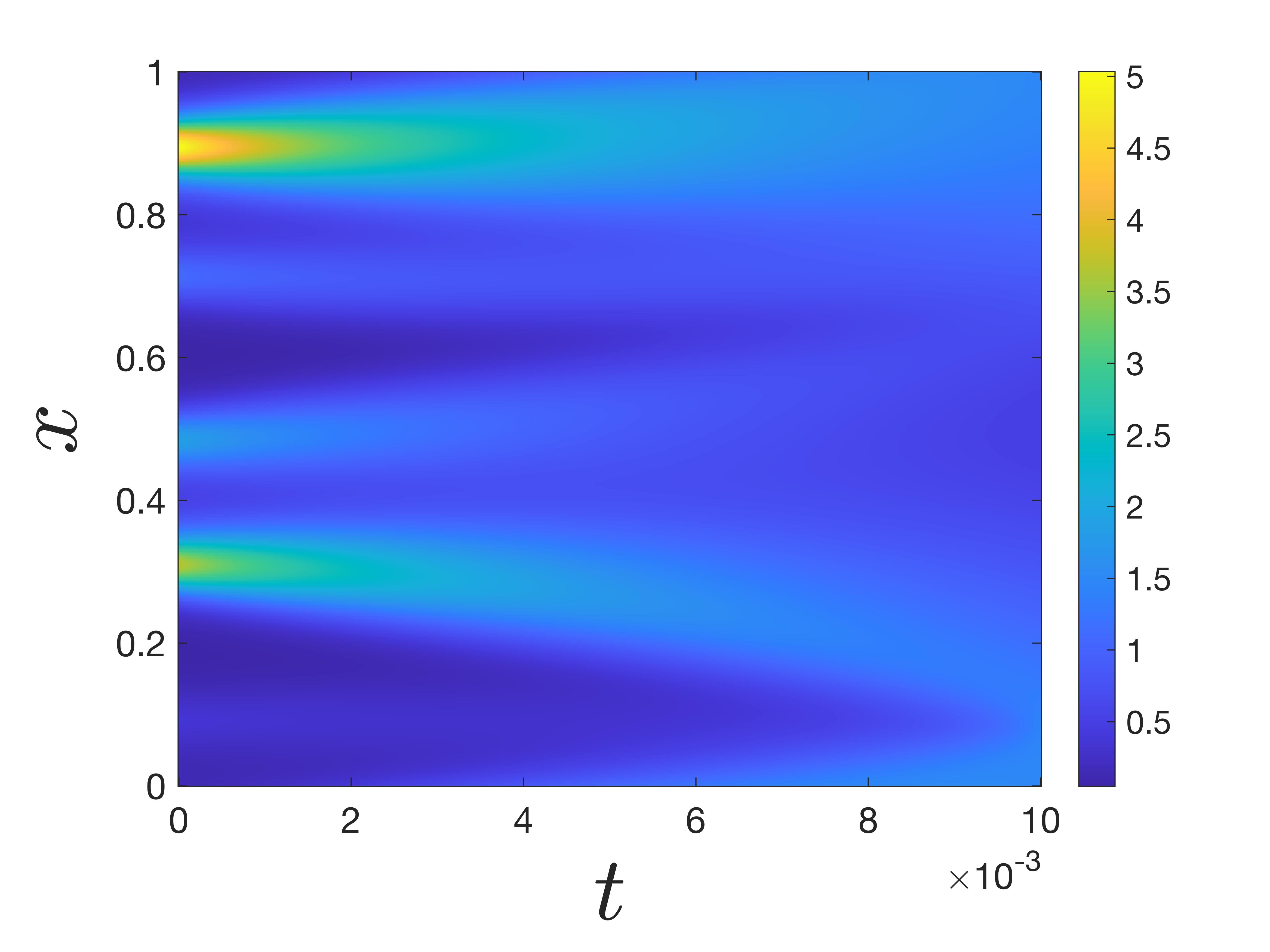} }
  \subfloat[$\nu = 0.1 $]{ \includegraphics[width=0.24\textwidth]{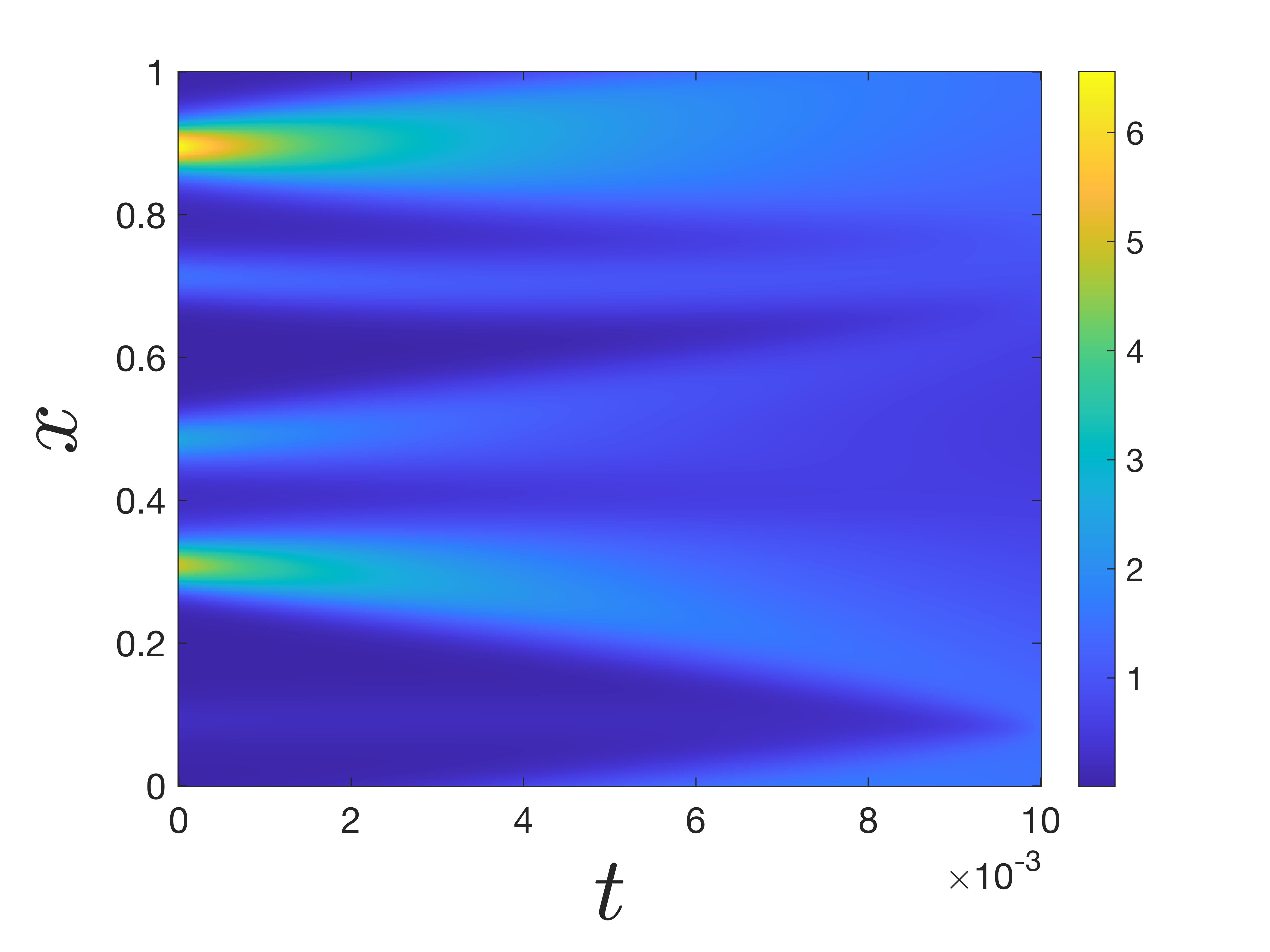} }
  \subfloat[$\nu = 0.05$]{ \includegraphics[width=0.24\textwidth]{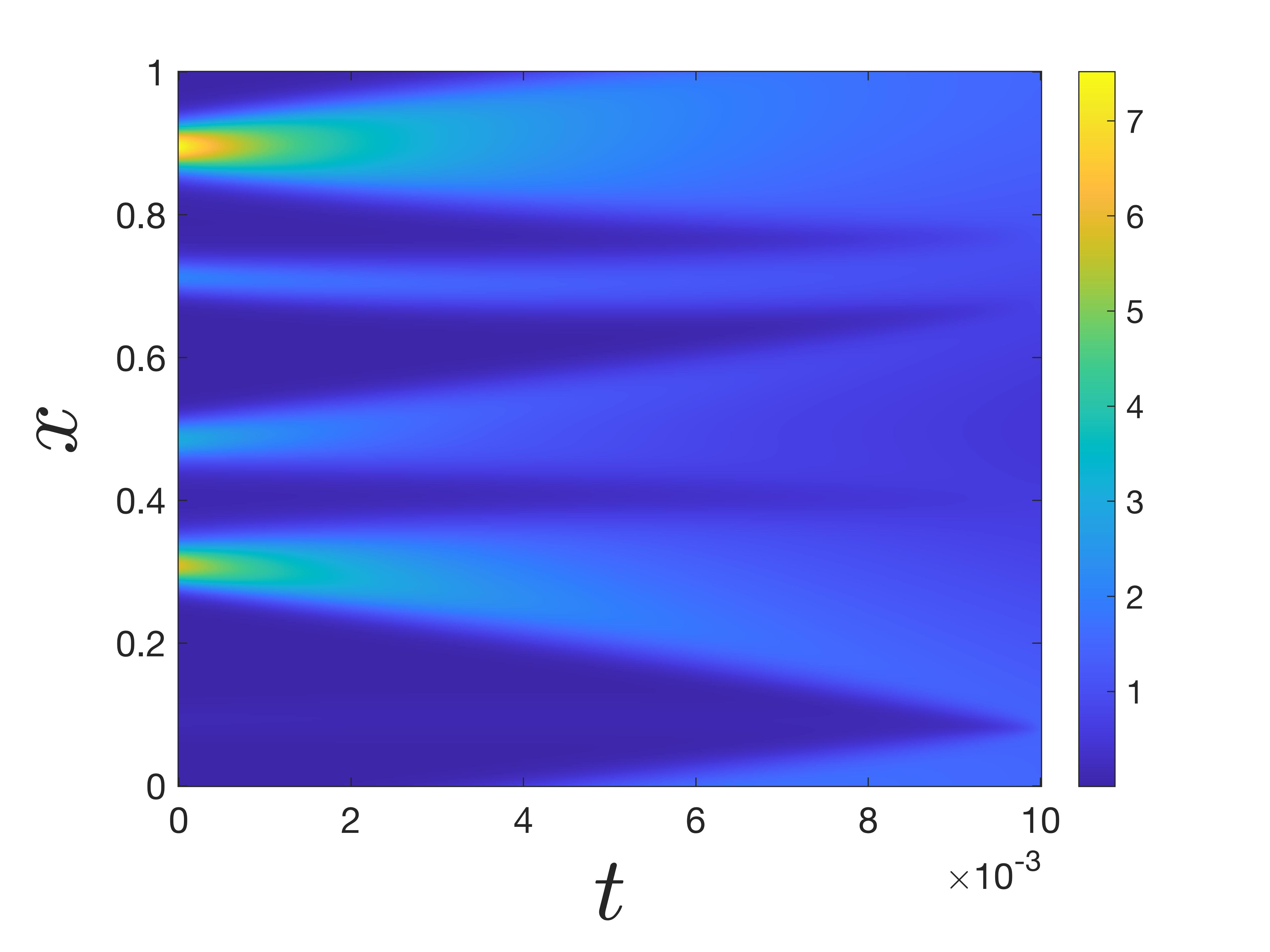} }
  \subfloat[$\nu = 0.02$]{ \includegraphics[width=0.24\textwidth]{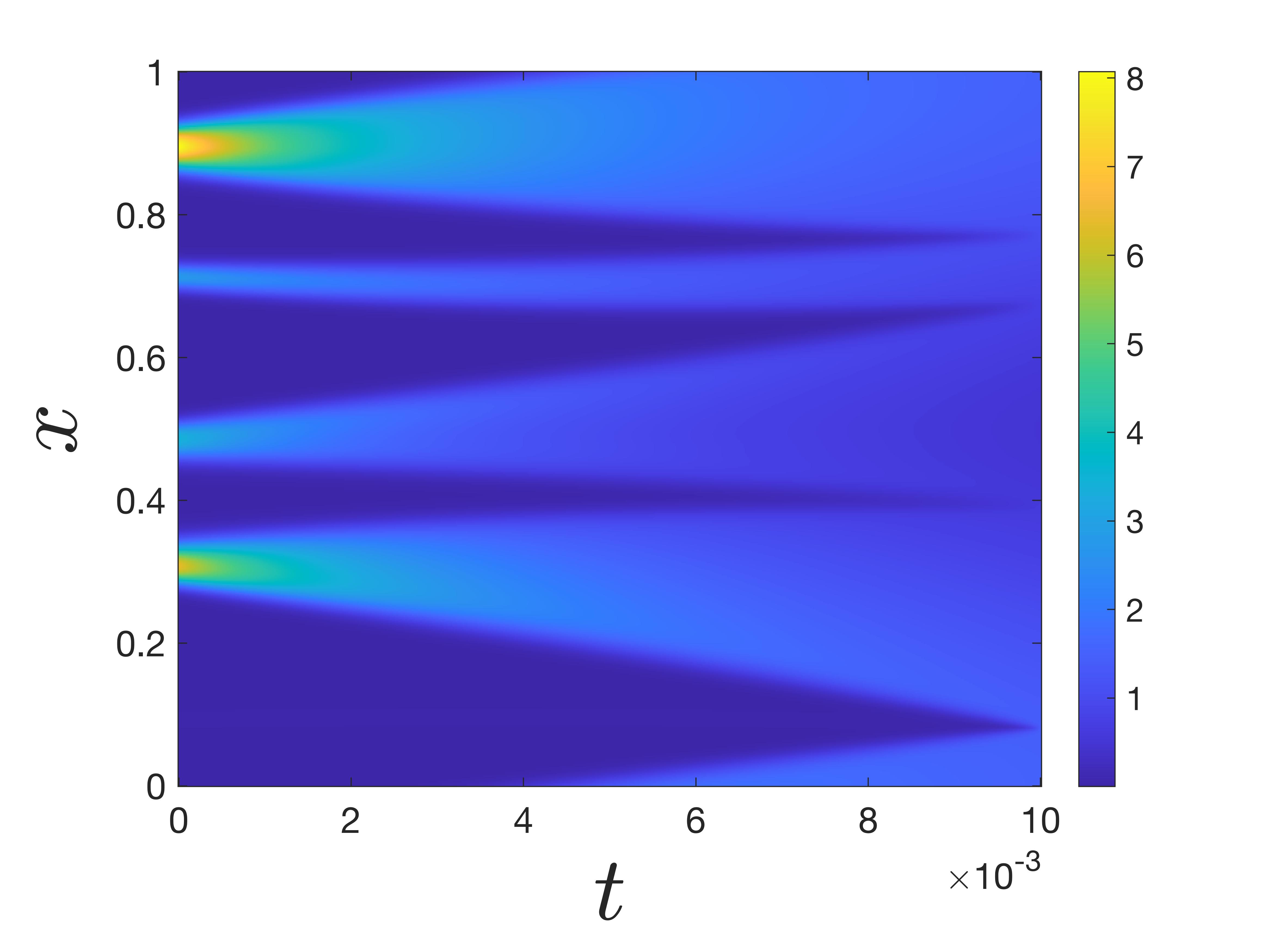} }
  \caption{ \label{fig:Solnu} Profile of the solution $M$ for different $\nu$.}
\end{figure}

\subsubsection{Nonlocal case}
In this section we test the performance of the proposed method on a nonlocal MFG problem. Inspired by the examples used in \cite{liu2020computational}, we consider $V[m](x) = \int_{0}^1 K(x,y)m(y)\mathrm{d}y$, where $K$ is symmetric. The initial-terminal conditions used are
\begin{subequations}\label{eq:bdconv}
  \begin{align}
    u_0(x) &= \cos(2\pi x), \quad x \in [0, 1], \\
    m_T(x) &= 1 + \frac{1}{2}\cos(2\pi x), \quad x \in [0, 1],
  \end{align}
\end{subequations}
and $T = 0.01$, $\nu = 0.4$, $\alpha = 0.1$ The results are presented in Figure
\ref{fig:Solconv}. By checking the local truncation error, we verified that the
computation converge. 

\begin{figure}[h]
  \centering
  \subfloat[$K(x,y)=900\sin(2\pi x)\sin(2\pi y)$]{ \includegraphics[width=0.45\textwidth]{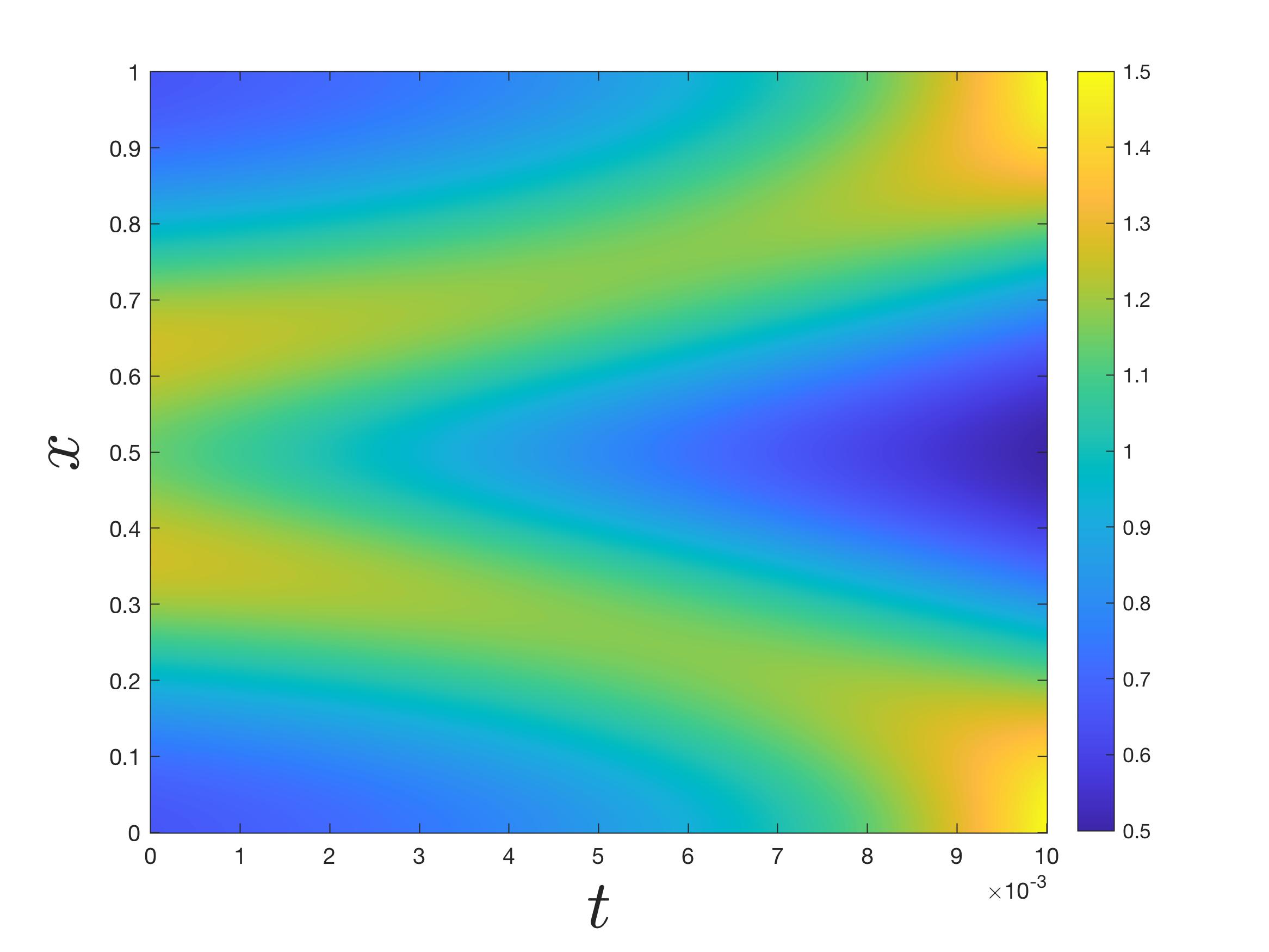} }
  \subfloat[$K(x,y)=2500\cos(6\pi x)\cos(6\pi y)$]{ \includegraphics[width=0.45\textwidth]{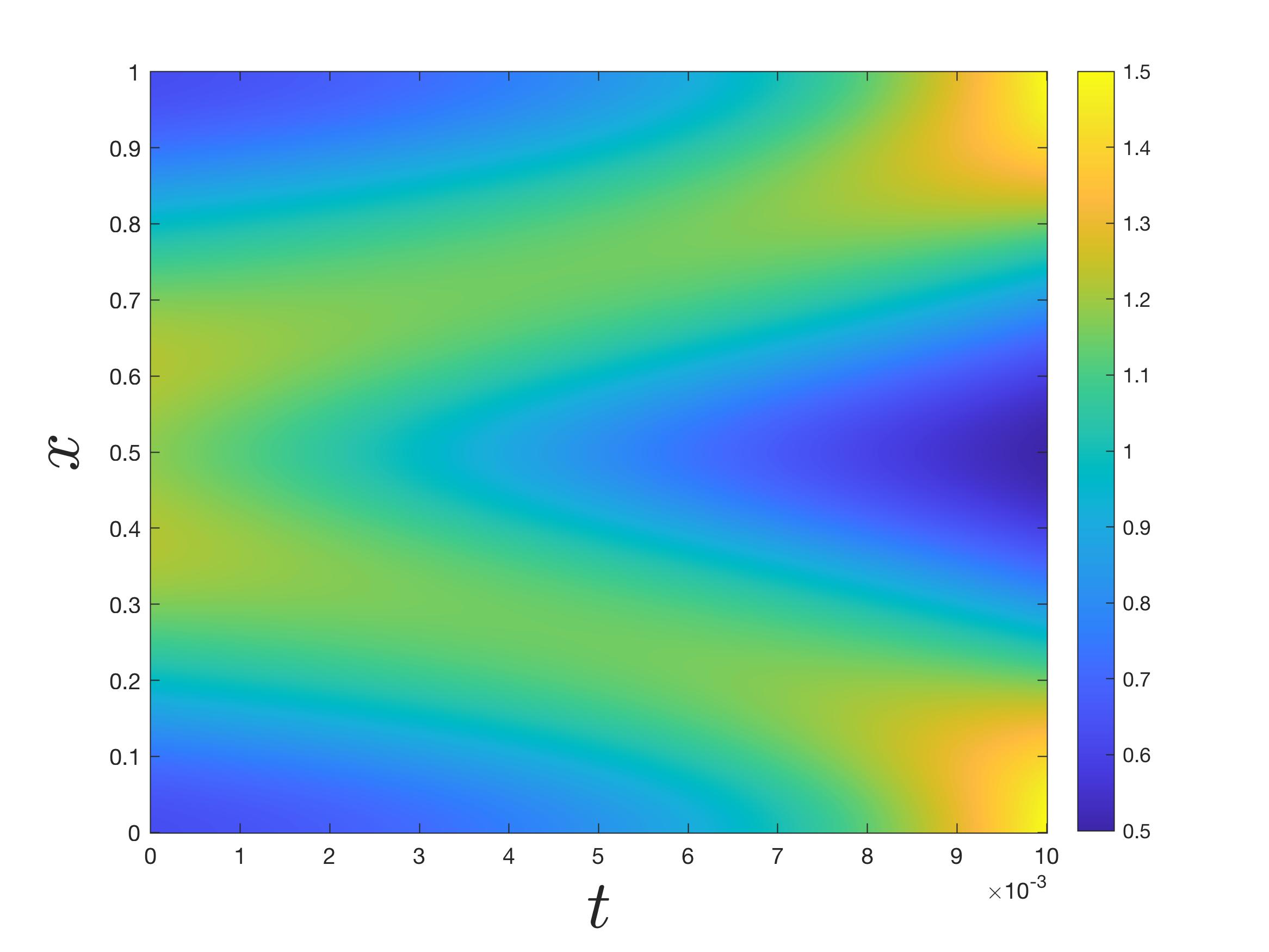} }
  \caption{\label{fig:Solconv} Profile of the solution $M$ for different nonlocal $V[m]$.}
\end{figure}

\subsection{Two-dimensional case}

In this subsection, we perform some tests for the two-dimensional case, and show that all the
conclusions for the one-dimensional case are also valid for the two-dimensional case.

\subsubsection{First-order scheme vs second-order scheme}\label{sec:1stvs2nd2d}
Let us first compare the performance between the first-order scheme and the second-order scheme
proposed in this paper. In two-dimensional cases, the cost of computation on the fine grid is
significantly higher than that on the coarse grid. Since the first-order scheme converges only by a
linear rate, one needs to compute the numerical solution on a very fine grid in order to get a
solution with high accuracy. The faster convergence of the second-order scheme becomes particularly
attractive. The following Hamiltonian is used in the test.
\begin{equation}
  H(x, \nabla u) = \cos(4\pi x_1) + \sin(2\pi x_1) + \sin(2\pi x_2) + |\nabla u|^{\gamma}, 
\end{equation}
with initial-terminal conditions given by
\begin{subequations}
  \begin{align}
    u_0(x) &= \cos(2\pi x_1)+\cos(2\pi x_2), \quad x \in [0, 1], \\
    m_T(x) &= 1 + \frac{1}{2}\cos(2\pi x_1)+\frac{1}{2}\cos(2\pi x_2), \quad x \in [0, 1].
  \end{align}
\end{subequations}
Other parameters and functions are chosen as:
\begin{equation}
  \nu = 1, \quad
  T = 1,\quad
  V[m](x) = m(x)^2, \quad
  \alpha = 1, \quad
  \gamma = 2,
\end{equation}
where $\alpha$ is the relaxation factor. 
With the goal of computing a numerical solution on the grid of size
$N_1^{(7)}=N_2^{(7)}=N_t^{(7)}=128$, we start the computation from the grid of the size
$N_1^{(4)}=N_2^{(4)}=N_t^{(4)}=16$.

Below we check the convergence orders of the two schemes, by calculating the relative error of
numerical solutions on each grid in comparison to the solution on the grid of the size
$N_1^{(8)}=N_2^{(8)}=N_t^{(8)}=256$.  \cref{tab:mfg2d_err} shows that the second-order scheme is truly of
second-order, and the first-order scheme is of first-order. The solution $M$ is shown in
\cref{fig:ctrM} at different times.

\begin{table}[ht]
  \centering
  \begin{tabular}{ccccccc}
    \hline
    $N_1$ & $N_2$ & $N_T$ & $err_U$, $err_M$ (1st)	& Order &  $err_U$, $err_M$ (2nd) & Order\\ \hline\hline
    16 & 16 & 16	&	2.6E-1,	6.5E-2 & --- &	1.0E0,	4.4E-2 & ---\\ \hline
    32 & 32 & 32	&	1.9E-1, 3.3E-2	& 0.49, 0.98 &	2.6E-1,	1.1E-2 & 1.9, 2.0\\ \hline
    64	& 64 &	64	&	1.0E-1, 1.4E-2	& 0.93, 1.2 &	6.4E-2,	2.4E-3 & 2.0, 2.2\\ \hline
    %        128	& 128 & 128	&	0.0405,	0.0045 &	0.0130,	0.0004\\ \hline
    \hline
  \end{tabular}
  \caption{\label{tab:mfg2d_err}The order of convergence of $U$ and $M$.}
\end{table}

\begin{figure}[h]
  \centering
  \subfloat[$t = 0      $]{ \includegraphics[width=0.31\textwidth]{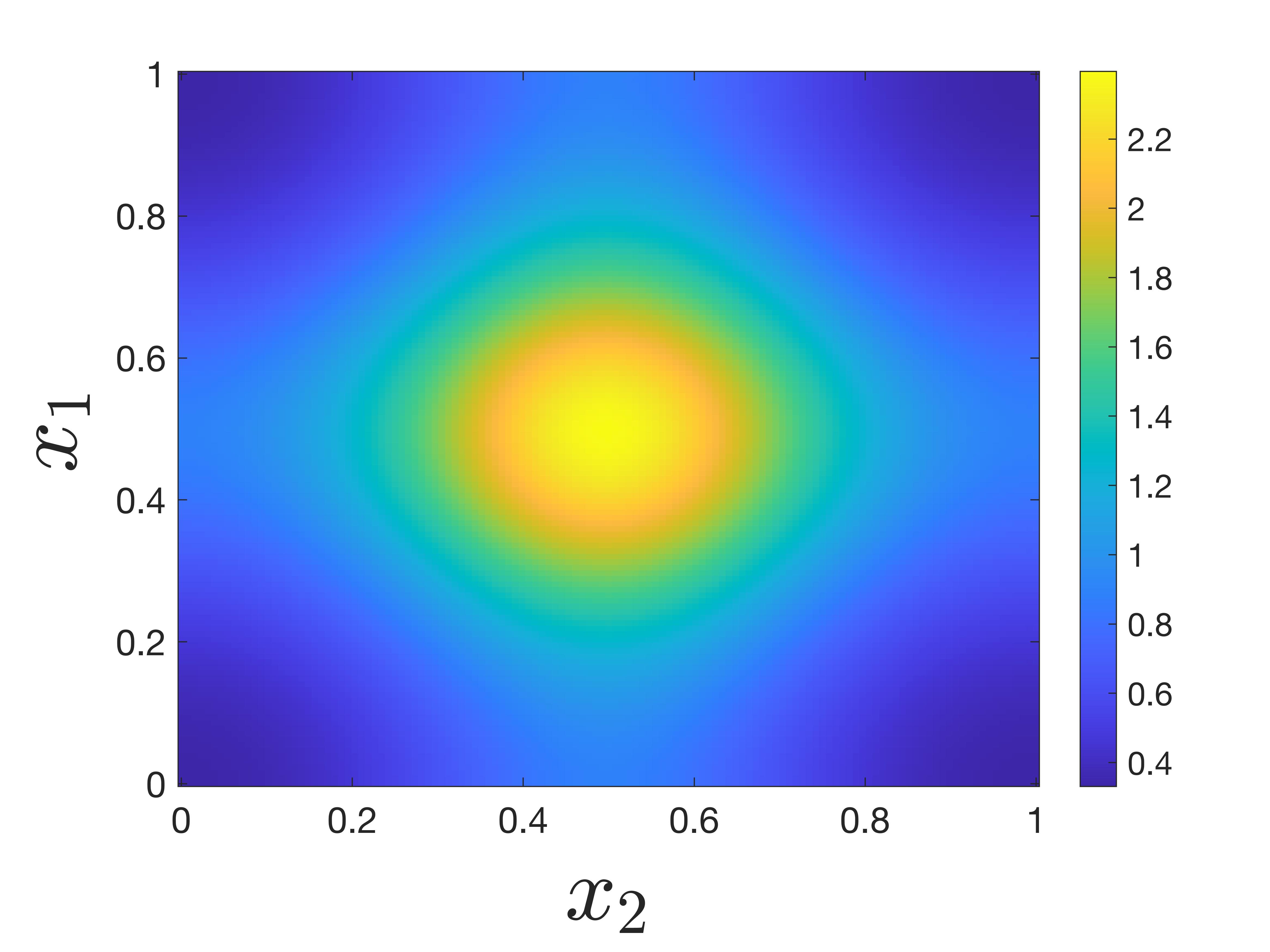} }
  \subfloat[$t = 7/128  $]{ \includegraphics[width=0.31\textwidth]{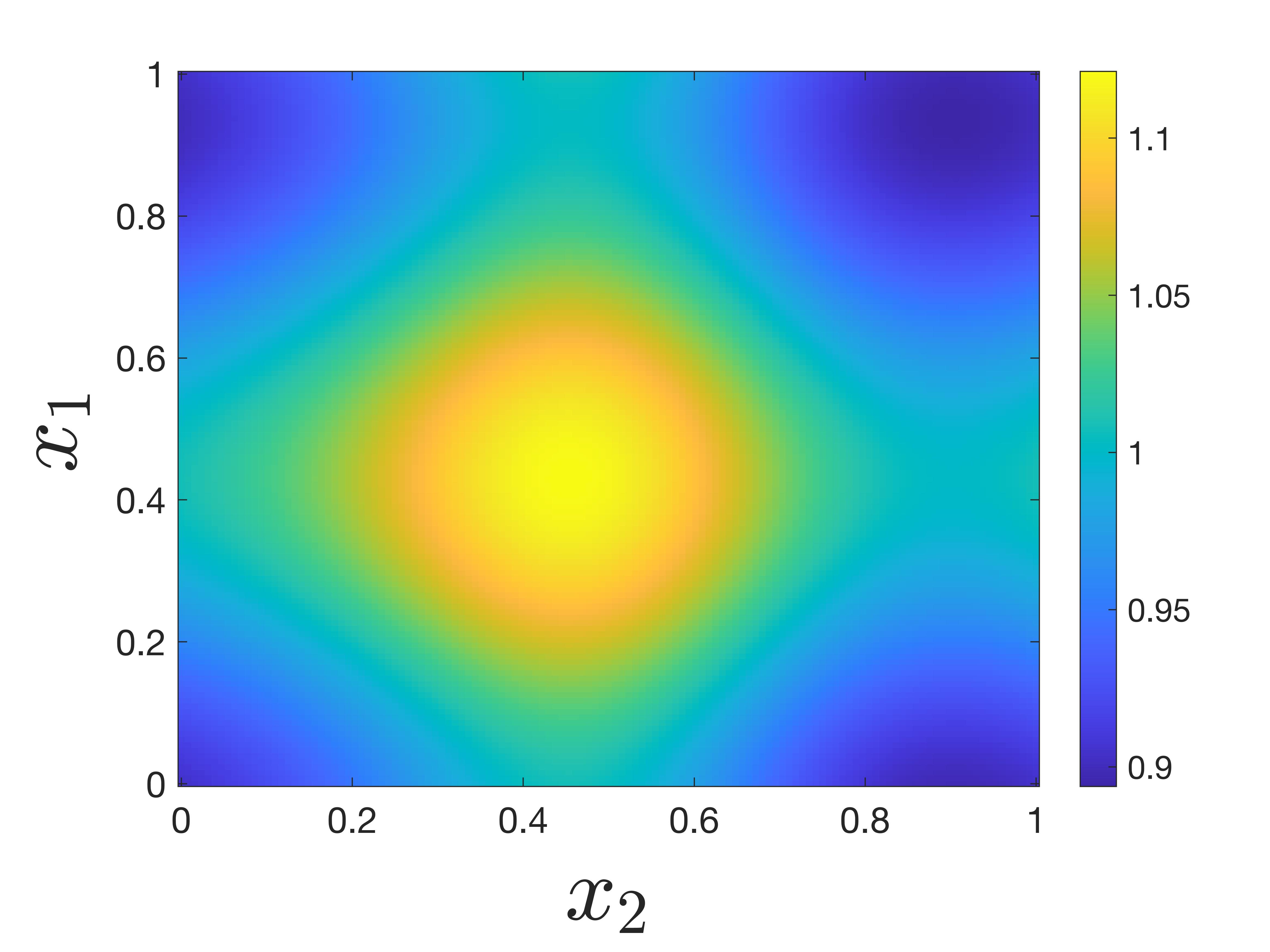} }
  \subfloat[$t = 11/128 $]{ \includegraphics[width=0.31\textwidth]{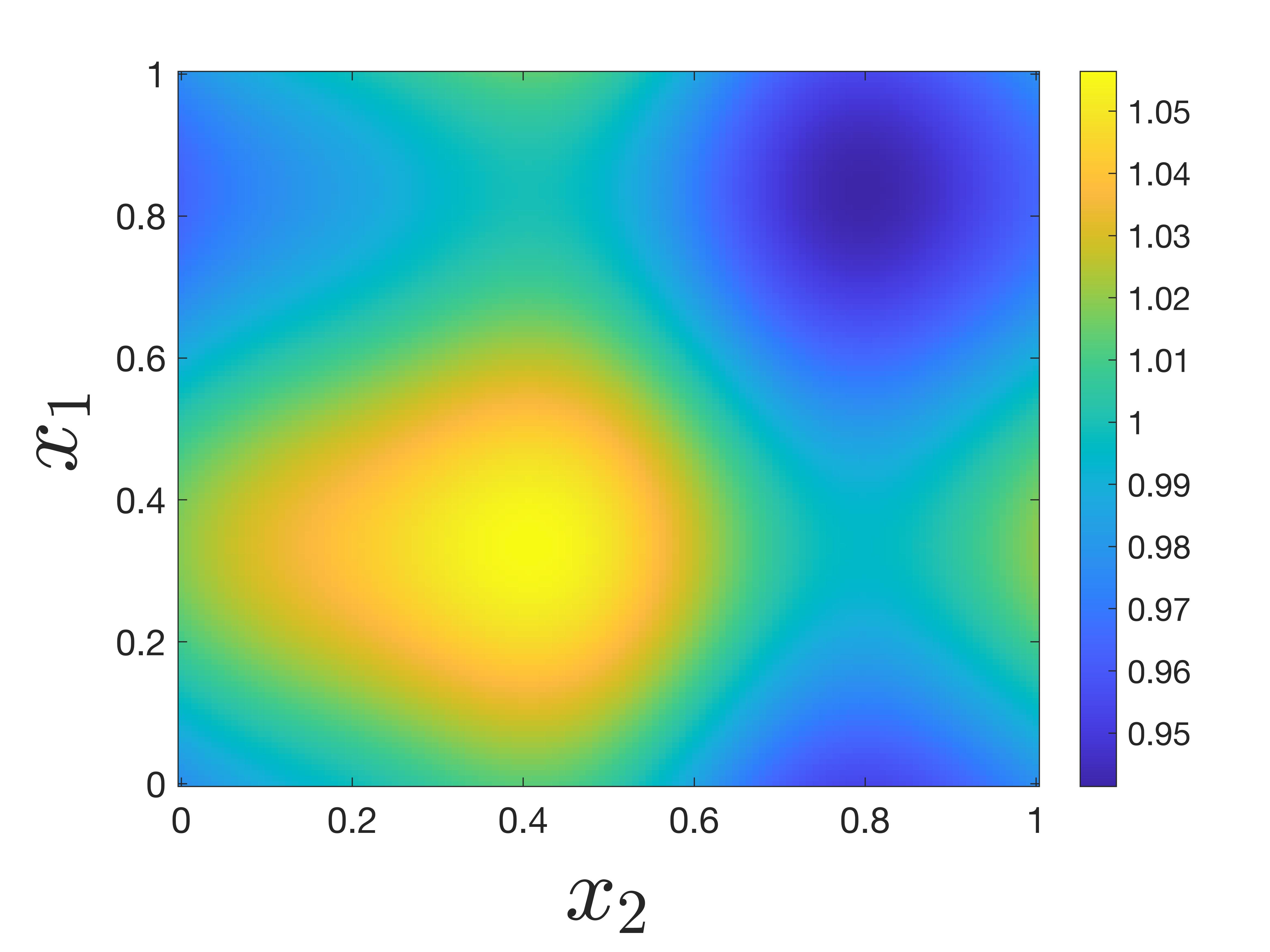} }\\
  \subfloat[$t = 15/128 $]{ \includegraphics[width=0.31\textwidth]{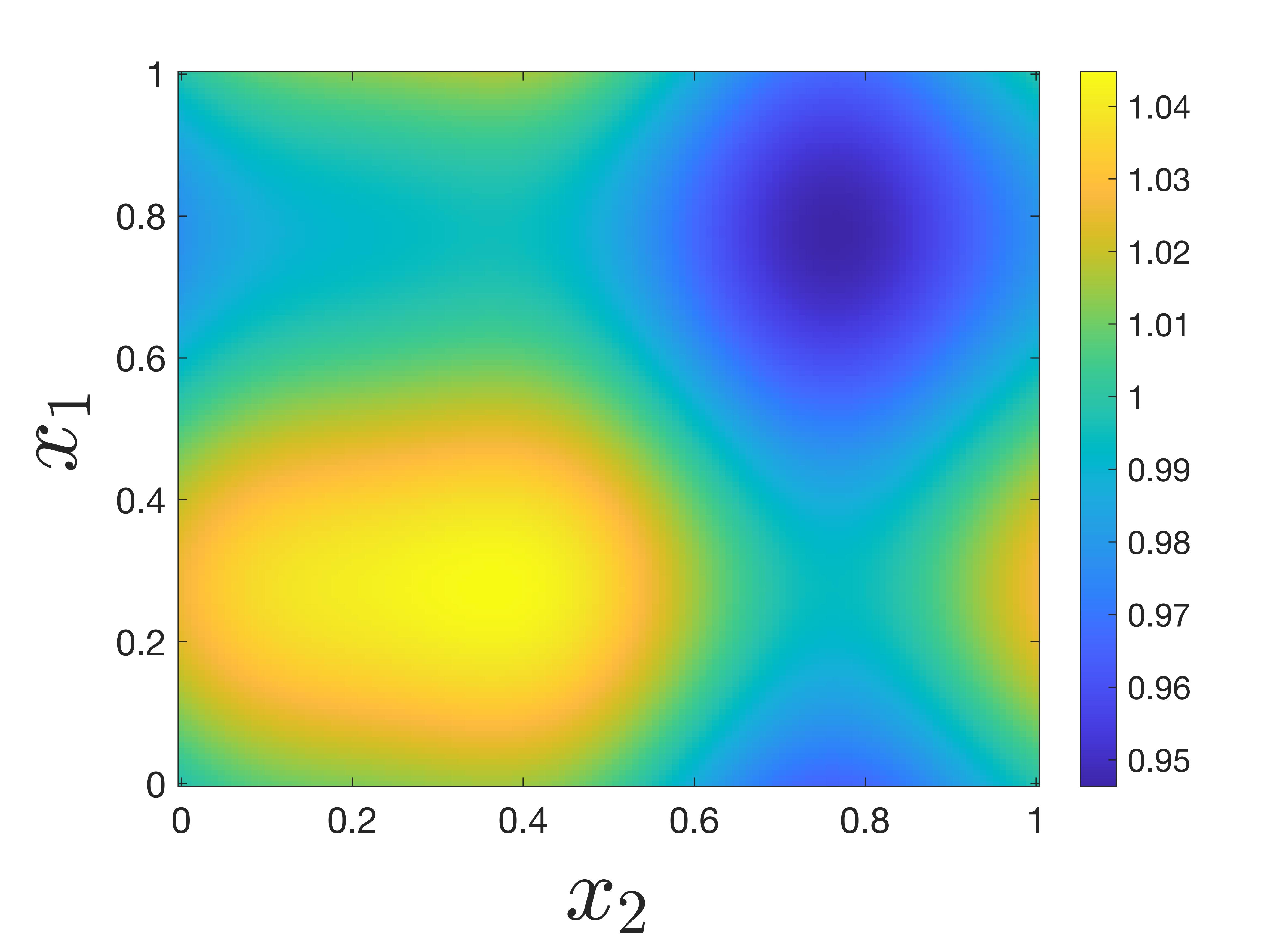} }
  \subfloat[$t = 19/128 $]{ \includegraphics[width=0.31\textwidth]{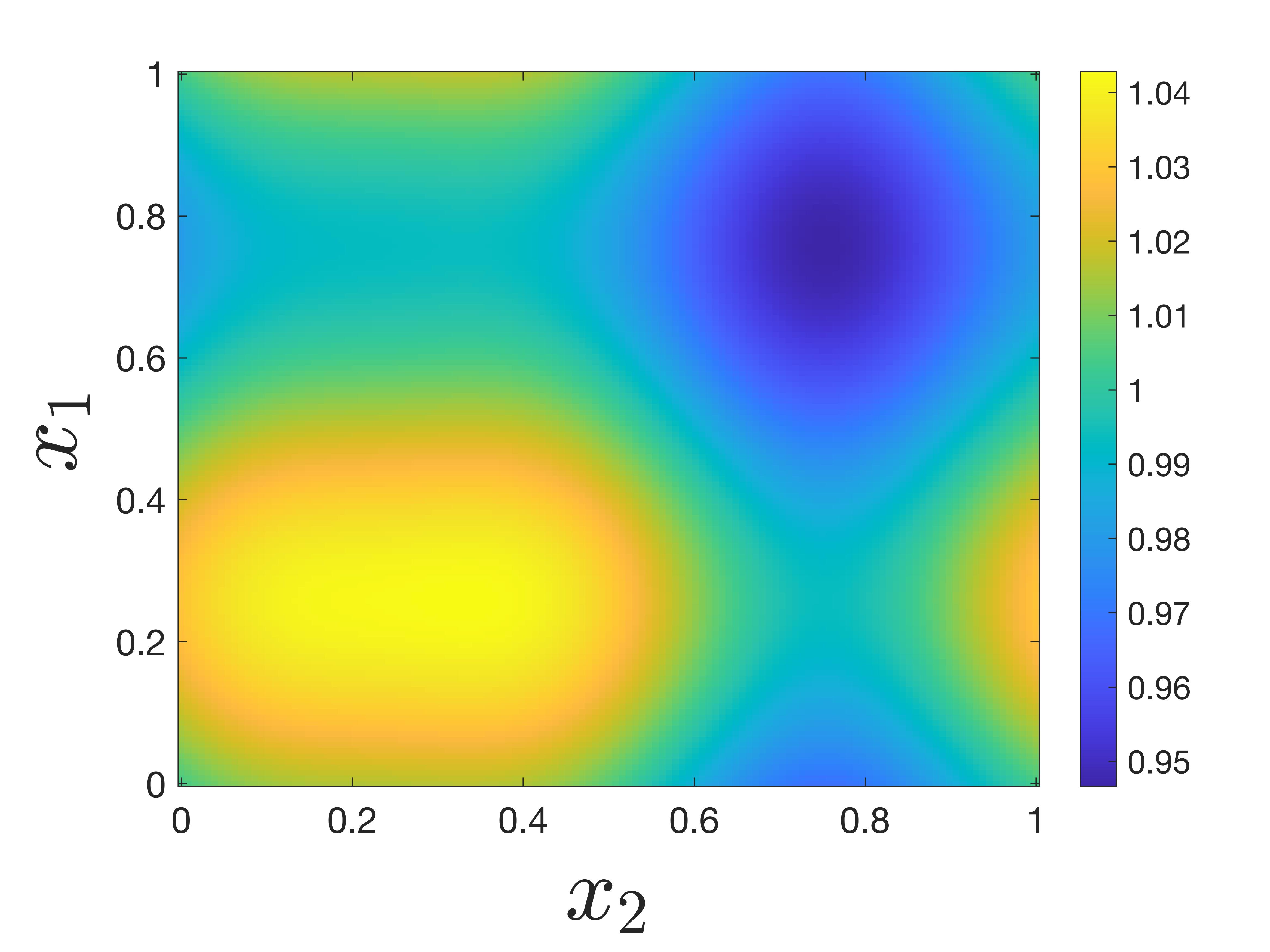} }
  \subfloat[$t = 111/128$]{ \includegraphics[width=0.31\textwidth]{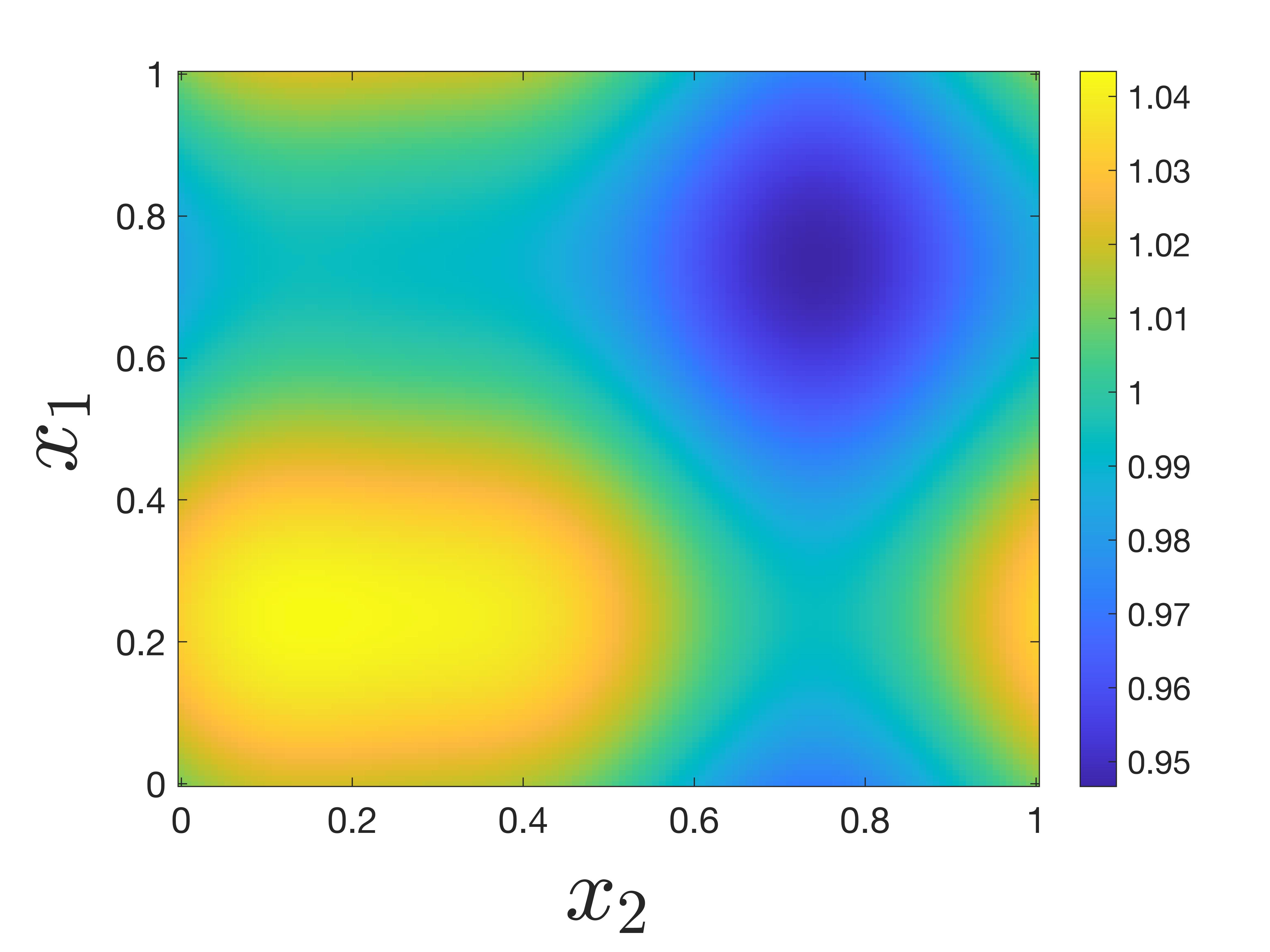} }\\
  \subfloat[$t = 115/128$]{ \includegraphics[width=0.31\textwidth]{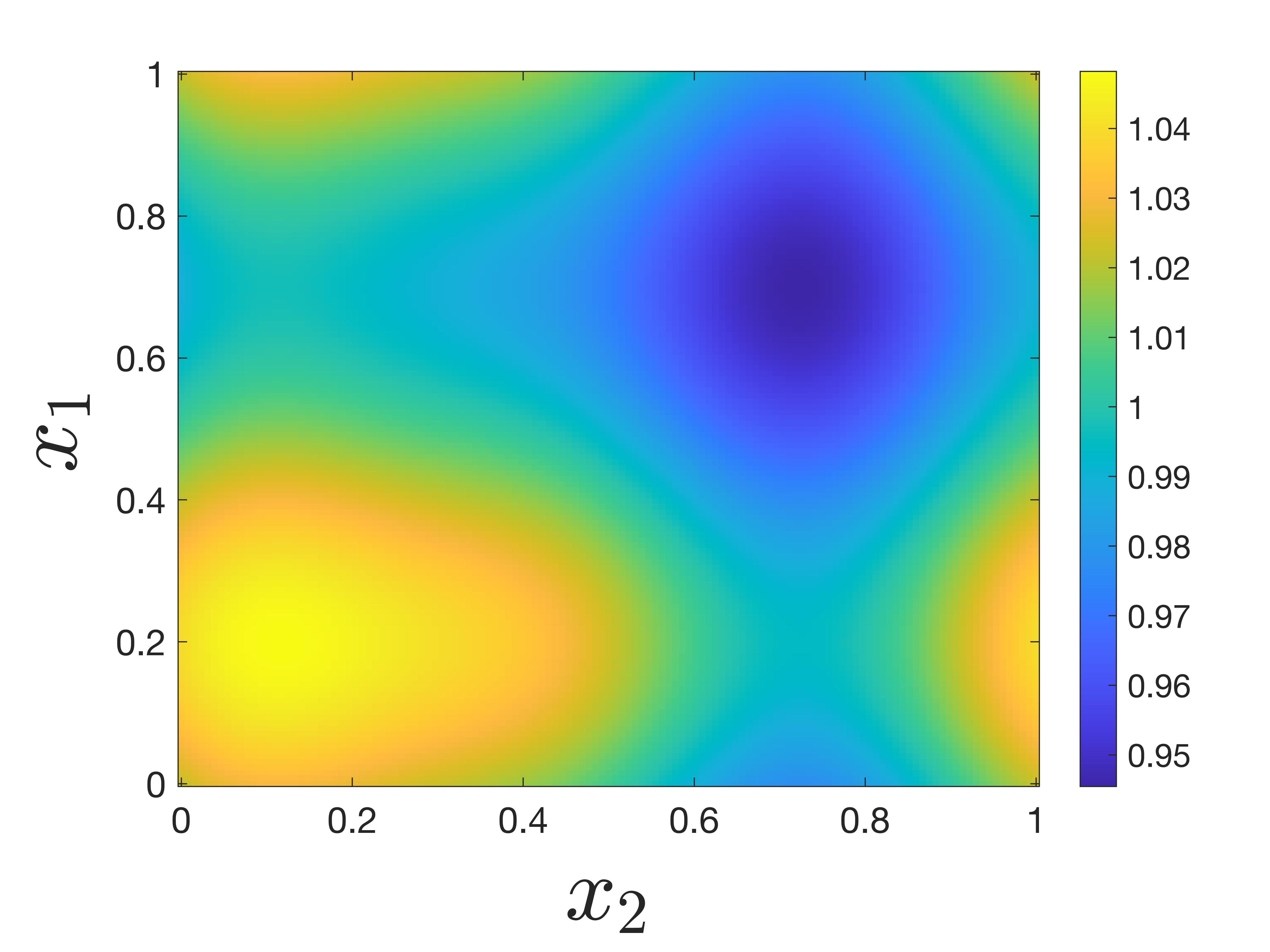} }
  \subfloat[$t = 120/128$]{ \includegraphics[width=0.31\textwidth]{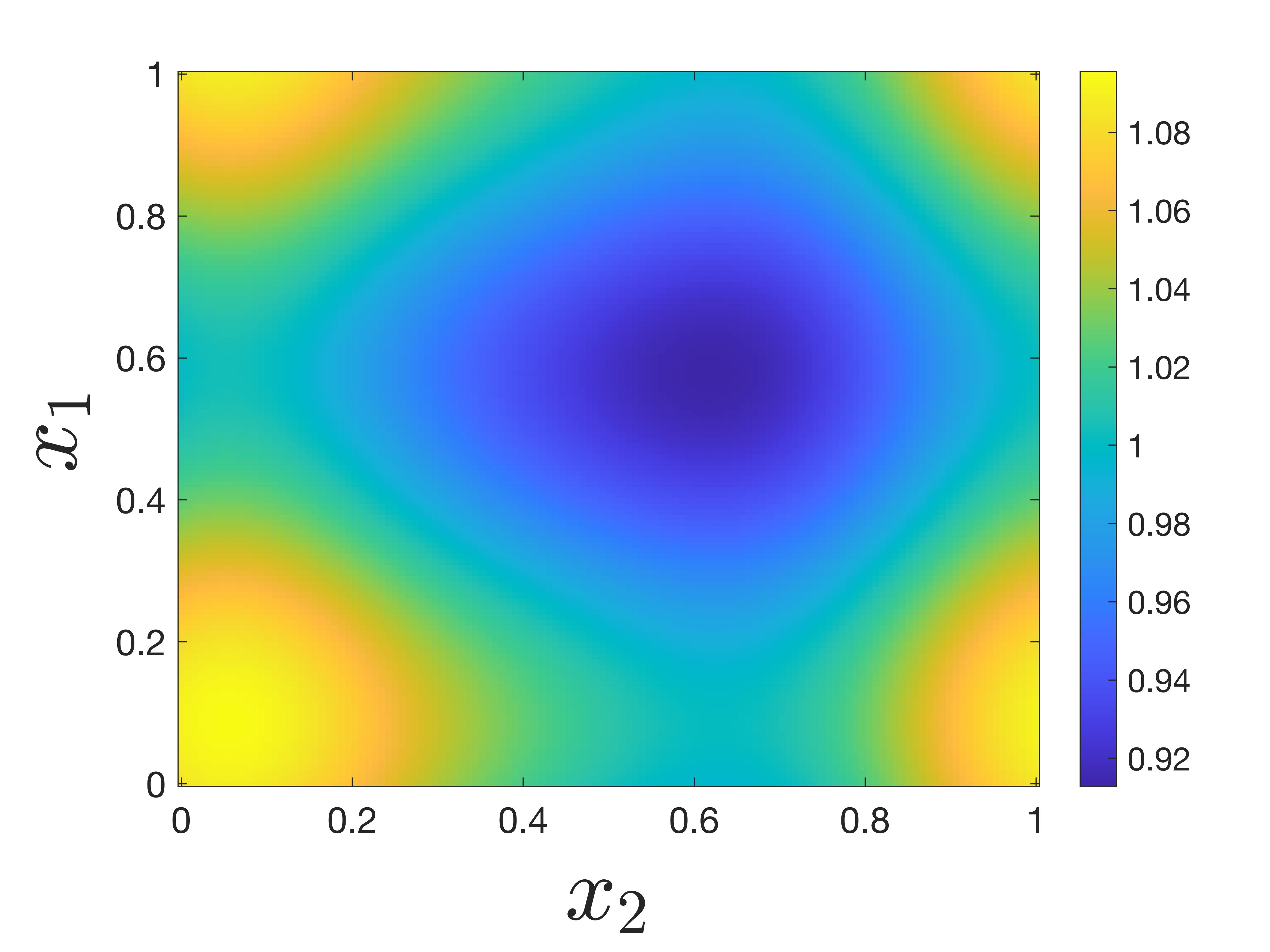} }
  \subfloat[$t = 128/128$]{ \includegraphics[width=0.31\textwidth]{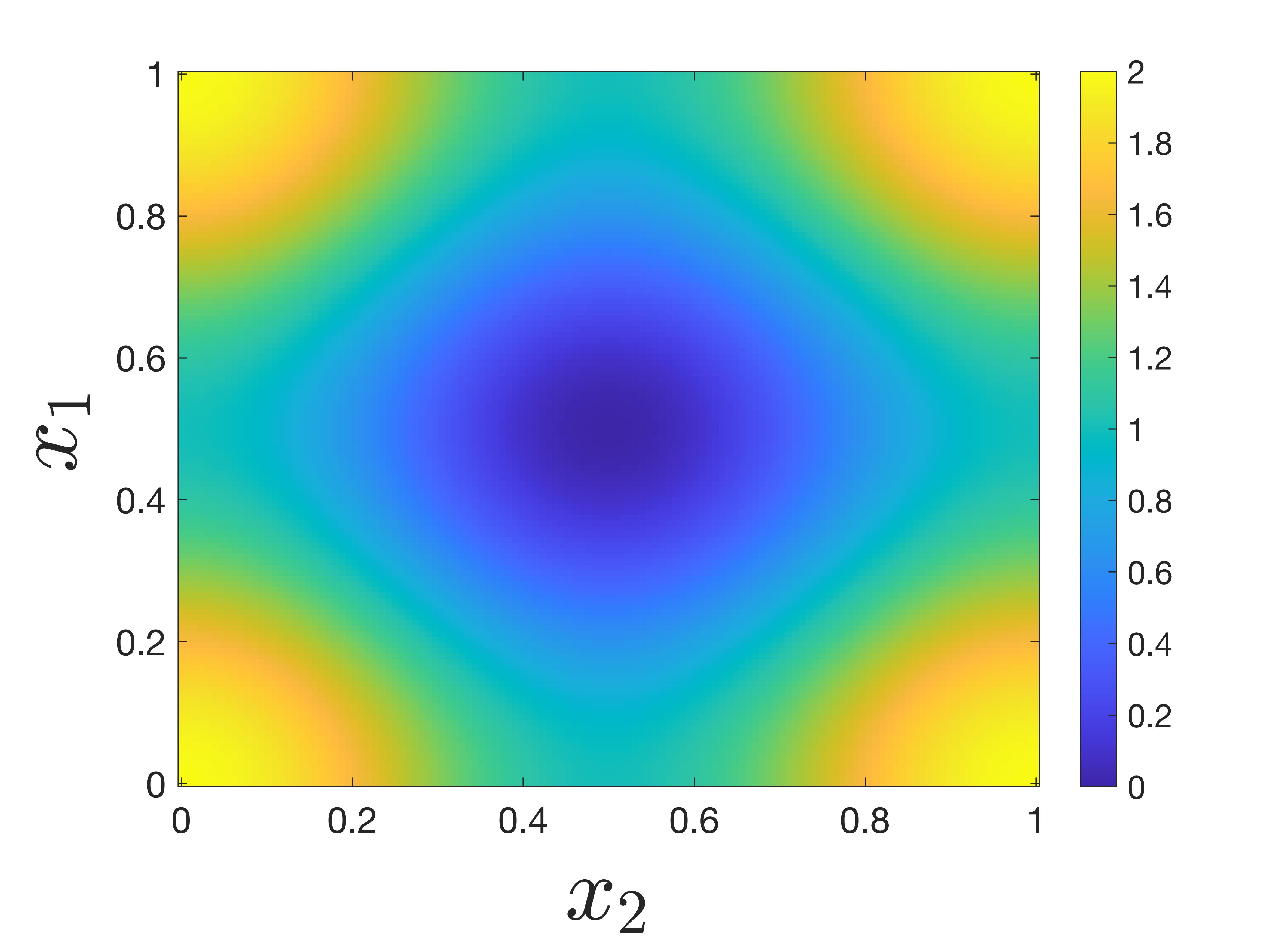} }
  \caption{ \label{fig:ctrM} Profile of the solution of $M$ on different time $t$. }
\end{figure}

\subsubsection{Multiscale vs. alternating sweeping}\label{sec:acc}

The following example provides a clear demonstration how the multiscale method can accelerate
convergence. The example used is the one in \cref{sec:1stvs2nd2d} with $\gamma = 1.5$ and
$\alpha=0.5$. The size of the finest grid is $N_1^{(L)}=N_2^{(L)}=N_t^{(L)}=320$. We compare the performance of using multiscale method with the performance of using only alternating sweeping. The alternating sweeping without multiscale takes $9$ steps to reach the accuracy desired with total time equal to $1.2\times10^4$ seconds. When applying the multiscale algorithm, the coarsest grid size is chosen as $N_1^{(L_0)}=N_2^{(L_0)}=N_t^{(L_0)}=20$ and it takes $1.6\times10^3$ seconds in total on the same machine. The acceleration ratio is equal to $1.2\times10^4/1.6\times10^3=7.5$.

\begin{figure}[h]
  \centering
  \includegraphics[width=0.5\linewidth]{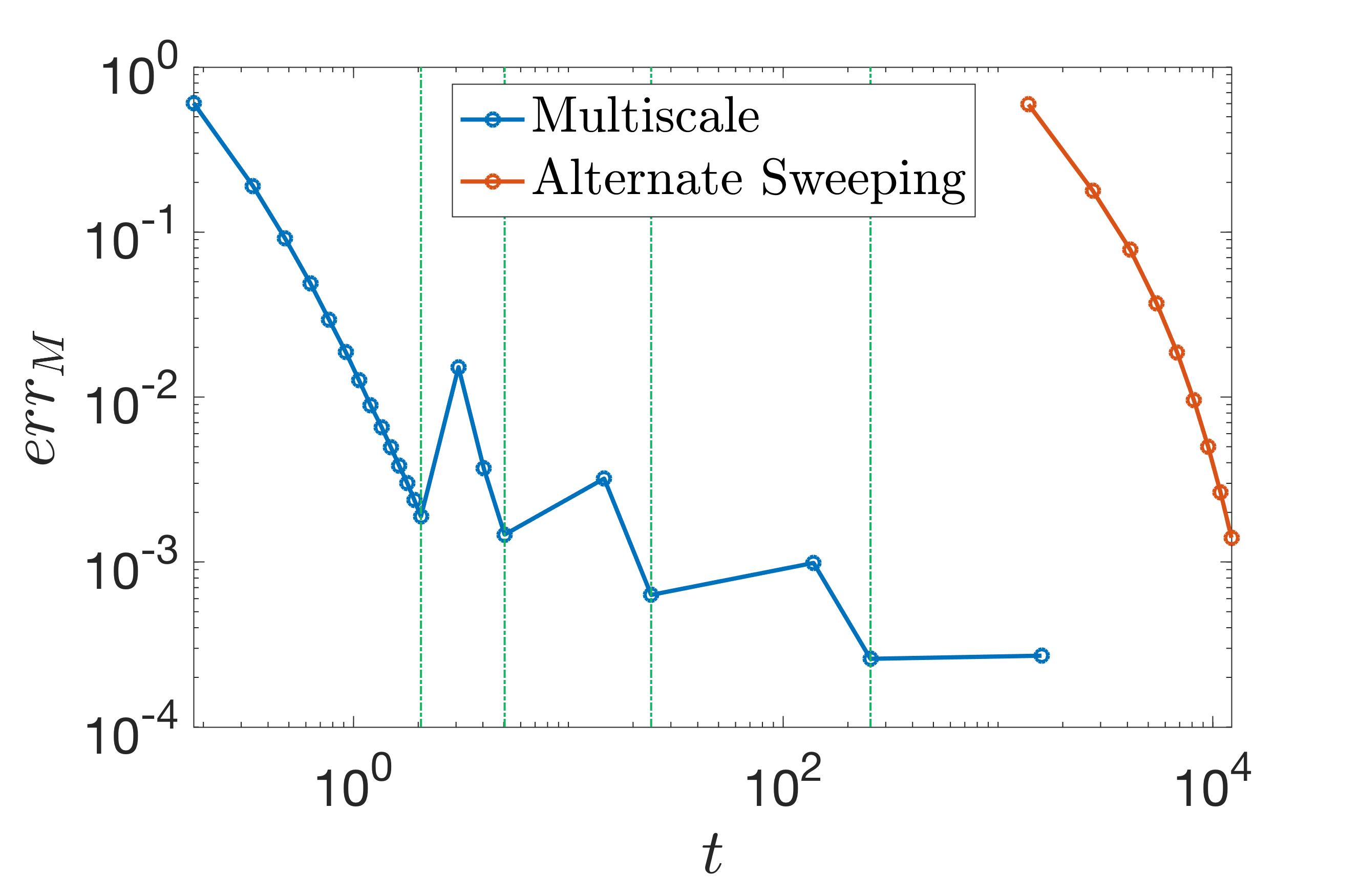}
  \caption{ The residual \eqref{fix} of the multiscale method and the alternating sweeping
algorithm.  Each point in the plot denotes the end time of an alternating sweeping step, thus the
plots do not start from $t=0$.}
  \label{fig:acc}
\end{figure}

The reason of this acceleration can be seen clearly in the computation process shown in
\cref{fig:acc}. The interpolation process introduces error, which is why the blue curve goes up at
the first step after each dash-dot line. As we can see, the time cost on the coarse grid is
negligible compared to the cost on the finest grid: although we need $14$ steps on the coarsest
grid, the cost of it is little, and we get a much better initial guess by the computation on coarse
grids. In the multiscale algorithm, the error decreases rapidly, leading to only $1$ alternating
sweeping on the finest grid after applying the multiscale method. In the case without multiscale,
the computation on the finest grid has to start from a much worse initial guess, leading it to take
a lot more steps to converge.

\section{Conclusion}\label{sec:conclusion}
We introduce an alternating sweeping method, which decouples the forward-backward MFG system into a
forward HJB system and a backward KFP system and allows for the use of classical time marching
numerical schemes on it. We also introduce a multiscale method along with relaxation technique in order
to guarantee the convergence.  A new second-order scheme is proposed for the time and spatial
discretization. Numerical results show the proposed multiscale algorithm with alternating sweeping
is robust and efficent in both one and two dimensions.

To simplify our discussion, we have only considered the periodic boundary conditions in space. An
immediate future work is to incorporate other boundary conditions, for example following the work of
\cite{Benamou2017}. With proper adaptation, the proposed method would also be applied to simliar
problems, such as the planning problems in \cite{achdou2012mean}.

\section*{Acknowledgements}
The work of Y.F. and L.Y. is partially supported by the U.S. Department of Energy, Office of
Science, Office of Advanced Scientific Computing Research, Scientific Discovery through Advanced
Computing (SciDAC) program and the National Science Foundation under award DMS-1818449.

\bibliographystyle{plain}
\bibliography{mfg}
\appendix
\section{Proof of \cref{lemma:specrad}}\label{proof:specrad}
Before the proof of \cref{lemma:specrad}, we introduce the 
famous spectral radius formula (Gelfand's formula \cite{Plax}) as follows.
\begin{lemma}
  For any element T of a Banach algebra, we have 
  \begin{equation}
    \rho(T) = \lim\limits_{n \rightarrow \infty} \|T^k\|^{1/k},
  \end{equation}
  where $\rho$ is the spectral radius. 
  And in particular, we have 
  \begin{equation}
    \rho(A) = \lim\limits_{n \rightarrow \infty} \|A^k\|^{1/k},
  \end{equation}
  for any matrix $A \in \mathbb{C}^{n\times n}$ and any matrix norm $\|\cdot\|$. 
\end{lemma}

\begin{proof}[Proof of \cref{lemma:specrad}]
  When $A=0$ or $B=0$, the conclusion holds trivially. From now on we assume that $A$ and $B$ are
  both non-zero. 

  For any matrix in a finite dimensional vector space, the matrix norm is always finite, and since
  $A$ and $B$ are non-zero, $0<\|A\| < \infty,~ 0<\|B\| < \infty$. Thus $\lim\limits_{n \rightarrow
  \infty} \|A\|^{\frac{1}{n}} = \lim\limits_{n \rightarrow \infty} \|B\|^{\frac{1}{n}} = 1$. Notice
  that
  \begin{equation}
    \|(AB)^n\|^{\frac{1}{n}} \leq \|A\|^{\frac{1}{n}} \|(BA)^{n-1}\|^{\frac{1}{n}} \|B\|^{\frac{1}{n}}.
  \end{equation}
  Since $\lim\limits_{n \rightarrow \infty} \|(BA)^{n-1}\|^{\frac{1}{n}} = \lim\limits_{n \rightarrow \infty} (\|(BA)^{n-1}\|^{\frac{1}{n-1}})^{\frac{n-1}{n}} = \rho(BA) $, $\lim\limits_{n \rightarrow \infty} \|A\|^{\frac{1}{n}} = \lim\limits_{n \rightarrow \infty} \|B\|^{\frac{1}{n}} = 1$, we have $\rho(AB)\leq\rho(BA)$ by taking the limit $n\rightarrow \infty$ in the inequality above. Put $A$ in the place of $B$ and $B$ in the place of $A$, we get $\rho(BA)\leq\rho(AB)$. Thus $\rho(AB)=\rho(BA)$.
\end{proof}

\section{Proof of \cref{prop:order}}\label{proof:order}

\begin{proof}[Proof of \cref{prop:order}]
  Supposing that $u$ and $m$ are the true solution of the equation \cref{eq:MFG}, we calculate the
  difference between the finite difference operator and the differential operator, and get
  \begin{equation}
    \frac{u(x_i, t_{n+1})-u(x_i, t_{n})}{\tau} - \pd{u}{t}(x_i, t_{n+\frac{1}{2}}) = -\frac{1}{12}\pdt{u}{t}{3}(x_i, t_{n+\frac{1}{2}}){\tau}^2 + O({\tau}^4),
  \end{equation}
  \begin{equation}
    \frac{u(x_{i+1}, t_n)-2u(x_i,t_n)+u(x_{i-1},t_n)}{h^2}-\pdt{u}{x}{2}(x_i,t_n) = -\frac{1}{12}\pdt{u}{x}{4}(x_i, t_n)h^2+O(h^4).
  \end{equation}
  Furthermore, we have:
  \begin{align*}
      \frac{3u(x_i)-4u(x_{i-1})+u(x_{i-2})}{2h} - \frac{\partial u}{\partial x}(x_i) &=  \frac{h^2}{3}\frac{{\partial}^3 u}{\partial x^3}(\xi), \quad \xi \in [x_{i-2}, x_i], \\
      \frac{-3u(x_i)+4u(x_{i+1})-u(x_{i+2})}{2h} - \frac{\partial u}{\partial x}(x_i) &=  \frac{h^2}{3}\frac{{\partial}^3 u}{\partial x^3}(\xi), \quad \xi \in [x_{i}, x_{i+2}],
  \end{align*}
  thus 
  \begin{equation*}
      \begin{aligned}
          \cD^+ u(x_i, t_n) - \pd{u}{x}(x_i, t_n) &= O(h^2), &
          \cD^- u(x_i, t_n) - \pd{u}{x}(x_i, t_n) &= O(h^2). 
      \end{aligned}
  \end{equation*}
  Thus we have
  \begin{equation*}
    \begin{aligned}
      \left|H(x_i, \cD u) - H(x_i, \pd{u}{x})\right| &\leq  \max\left(\left|\pd{H}{p}(x_i, \pd{u}{x})(\cD^+u-\pd{u}{x})\right|, \left|\pd{H}{p}(x_i, \pd{u}{x})\cdot(\cD^-u-\pd{u}{x})\right|\right) + O(h^6) \\
      &= O(h^2).
    \end{aligned}
  \end{equation*}
  Because $V$ is defined pointwisely, operator $V$ has no truncation error. The analysis for the Fokker-Planck equations is the same. 
  Combine the formula above, we get the desired conclusion. 
\end{proof}

\section{Proof of \cref{prop:mass}}\label{proof:mass}
\begin{proof}[Proof of \cref{prop:mass}]
  The total mass can be represented as:
  \begin{equation}
    total ~ mass = h\sum_{i}{m_i} = h\innerproduct{m}{1}{h}. 
  \end{equation}
  From \eqref{KFPfd} we know: 
  \begin{equation*}
    {m^n-m^{n+1}} = \frac{\tau}{2}(-\cL (m^n + m^{n+1}) + B(m^n, u^n) + B(m^{n+1}, u^{n+1})), \quad n = 0, 1, 2, \ldots, N_t - 1, 
  \end{equation*}
  thus 
  \begin{equation*}
    \begin{aligned}
      \innerproduct{m^n}{1}{h} - \innerproduct{m^{n+1}}{1}{h} & = 
      -\frac{ \tau}{2}\innerproduct{\cL(m^n + m^{n+1})}{1}{h} + \frac{\tau}{2}(\innerproduct{B(m^n, u^n)}{1}{h} + \innerproduct{B(m^{n+1}, u^{n+1})}{1}{h})\\
      & = -\frac{\nu \tau}{2}\innerproduct{\cD_0(m^n + m^{n+1})}{\cD_0 1}{h} - \frac{\tau}{2}(\innerproduct{m^n\nabla_p \cH^n}{\cD 1}{h} + \innerproduct{m^{n+1}\nabla_p H^{n+1}}{\cD 1}{h})\\
      & = 0,  \quad n = 0, 1, 2, \ldots, N_t - 1, 
    \end{aligned}
  \end{equation*}
  where $\cD_0 W = \left({\frac{W_{i+1} - W_{i}}{h}, \frac{W_{i} - W_{i-1}}{h}} \right)$, and $1$ denotes the grid function whose every coorinate equals $1$, and the last equation is correct because every coordinate of $\cD 1$ and $\cD_0 1$ is $0$. 

  As a result, 
  \begin{equation}
    h\innerproduct{m^n}{1}{} = h\innerproduct{m^{N_t}}{1}{} = 1,  
  \end{equation}
  which is exactly the conclusion we want to draw. 
\end{proof}

\end{document}